\documentclass[11pt,reqno,tbtags]{amsart}

\usepackage{amsthm,amssymb,amsfonts,amsmath} 
\usepackage{amscd} 
\usepackage{mathrsfs} 
\usepackage{mathtools} 
\usepackage{bbm} 
\usepackage[numbers, sort&compress]{natbib} 
\usepackage{subfigure, graphicx} 
\usepackage{vmargin} 
\usepackage{setspace} 
\usepackage{paralist} 
\usepackage{hyperref} 
\usepackage{stmaryrd} 
\usepackage{float} 
\numberwithin{equation}{section} 
\usepackage{titlesec} 
\titleformat{\section}[block]{\large\bfseries\filcenter}{\thesection.}{1em}{} 
\titleformat{\subsection}[block]{\bfseries}{\thesubsection.}{1em}{} 

\makeatletter
\def\l@section{\@tocline{1}{0pt}{0pc}{1pc}{}}
\def\l@subsection{\@tocline{2}{0pt}{1pc}{2pc}{}}
\makeatother

\providecommand{\R}{}

\providecommand{\N}{}

\renewcommand{\R}{\mathbb{R}}

\renewcommand{\N}{{\mathbb N}}

\newcommand{\E}[1]{{\mathbf E}\left[#1\right]}

\newcommand\cD{\mathcal D}

\newcommand\cF{\mathcal F}
\newcommand\cG{\mathcal G}

\newcommand\cI{\mathcal I}
\newcommand\cJ{\mathcal J}

\newcommand\cN{\mathcal N}

\newcommand\cP{\mathcal P}

\newcommand\cS{{\mathcal S}}

\newcommand\cU{{\mathcal U}}

\newcommand{\bE}{\mathbf{E}}

\newcommand{\bP}{\mathbf{P}}

\newcommand{\convdist}{\ensuremath{\stackrel{\mathrm{d}}{\rightarrow}}}
\newcommand{\convas}{\ensuremath{\stackrel{\mathrm{a.s.}}{\rightarrow}}}
\newcommand{\aseq}{\ensuremath{\stackrel{\mathrm{a.s.}}{=}}}

\newcommand{\pran}[1]{\left(#1\right)}
\providecommand{\eps}{}
\renewcommand{\eps}{\epsilon}
\providecommand{\ora}[1]{}
\renewcommand{\ora}[1]{\overrightarrow{#1}}
\newcommand\urladdrx[1]{{\urladdr{\def~{{\tiny$\sim$}}#1}}} 
\DeclareRobustCommand{\SkipTocEntry}[5]{} 

\newtheorem{thm}{Theorem}[section]
\newtheorem{lem}[thm]{Lemma}
\newtheorem{prop}[thm]{Proposition}
\newtheorem{cor}[thm]{Corollary}
\newtheorem{dfn}[thm]{Definition}

\newtheorem{fact}[thm]{Fact}
\newtheorem{claim}[thm]{Claim}

\newtheorem{remark}[thm]{Remark}

\numberwithin{thm}{section}
\newcounter{atscan}
\renewcommand{\theatscan}{$A_{t,\arabic{atscan}}$}
\newcounter{btscan}
\renewcommand{\thebtscan}{$B_{t,\arabic{btscan}}$}

\usepackage{color} 
\reversemarginpar
\definecolor{clou}{rgb}{0.8,0.25,0.5125}
\newcommand{\lou}[1]{\textcolor{cjl}{#1}}

\definecolor{cjl}{rgb}{0,0,0}
\newcommand{\jl}[1]{\textcolor{cjl}{#1}}

\newcommand{\rX}{\ensuremath{\mathrm{X}}}

\newcommand{\rx}{\ensuremath{\mathrm{x}}}
\newcommand{\rW}{\ensuremath{\mathrm{W}}}

\newcommand{\rB}{\ensuremath{\mathrm{B}}}

\newcommand{\argmax}{\ensuremath{\mathop{\mathrm{argmax}}}}
\newcommand{\argmin}{\ensuremath{\mathop{\mathrm{argmin}}}}
\DeclareRobustCommand{\stirling}{\genfrac\{\}{0pt}{}}

\newcommand{\rad}{\ensuremath{r_{N}}}

\begin{document}

\title{Barycentric Brownian Bees} 
\author{Louigi Addario-Berry}
\author{Jessica Lin}
\author{Thomas Tendron}

\date{June 5, 2020; revised August 16, 2021} 

\address[LAB]{Department of Mathematics and Statistics, McGill University, Burnside Hall, 805 Sherbrooke Street West, Montreal, QC H3A 0B9}
\email{louigi.addario@mcgill.ca}
\urladdrx{http://problab.ca/louigi/}

\address[JL]{Department of Mathematics and Statistics, McGill University, Burnside Hall, 805 Sherbrooke Street West, Montreal, QC H3A 0B9}
\email{jessica.lin@mcgill.ca}
\urladdrx{https://sites.google.com/view/jessicalin-math/home}

\address[TT]{Department of Statistics, University of Oxford, 24 St Giles’, Oxford, United Kingdom OX1 3LB}
\email{thomas.tendron@mail.mcgill.ca}
\urladdrx{https://thomastend.github.io/ttend/}


\begin{abstract} 
We establish an invariance principle for the barycenter of a Brunet-Derrida particle system in $d$ dimensions. The model consists of $N$ particles undergoing dyadic branching Brownian motion with rate 1. At a branching event, the number of particles is kept equal to $N$ by removing the particle located furthest away from the barycenter. To prove the invariance principle, a key step is to establish Harris recurrence for the process viewed from its barycenter.
\end{abstract}

\maketitle
\tableofcontents



%



%

\section{Introduction} \label{sec:intro}
The {\em barycentric Brownian bees} processes are Brunet--Derrida particle systems defined as follows. The population consists of $N$ individual particles, whose positions are points in $\R^d$. Independently, each particle moves according to a standard Brownian motion, and undergoes binary branching at rate one. The instant when a branching event occurs, the particle furthest from the current barycenter of the particles is removed. Our notation for such a process is 
\[
\rX = (X(t),t \ge 0) = ((X_i(t))_{i\in [N]},t \ge 0),  
\]
\jl{where $[N]:=\left\{1, 2, \ldots, N\right\}$ denotes the index set of natural numbers up to $N$.} We write $\overline{X}(t) = N^{-1}\sum_{1 \le i \le N} X_i(t)$ and $\overline{\rX}=(\overline{X}(t),t \ge 0)$, and call $\overline{X}(t)$ the {\em barycenter} of $X(t)$. 

The result of the current work is an invariance principle for the barycenter process $\overline{\rX}$. 
\begin{thm}\label{thm:main}
For all $d \ge 1$ and $N \ge 1$, there exists $\sigma=\sigma(d,N)\in (0, \infty)$ such that, as $m \to \infty$,
\[
\pran{m^{-1/2}\overline{X}(tm),0 \le t \le 1} \convdist  (\sigma B(t),0 \le t \le 1),
\]
with respect to the Skorohod topology on $\cD([0,1], \R^d)$, where $(B(t),0 \le t \le 1)$ is a standard Brownian motion in $\R^d$ starting at the origin. 
\end{thm}
It turns out that $\sigma(d, 1)=1=\sigma(d,2)$ for all $d\geq 1$. For $N\geq 3$, our proof does not yield insight into the value of $\sigma(d,N)$. We point out that whatever the initial configuration $X(0),$ it is sent to the origin by this scaling. Since the Skorohod topology relativized to $C([0,1], \R^d)$ coincides with the uniform topology, it is not hard to see that any reasonable smoothing of $\left(m^{-1/2}\overline{X}(tm), 0\leq t\leq 1\right)$ will converge in distribution to $(\sigma B(t),0 \le t \le 1)$ in the uniform topology on $C([0,1], \R^d)$. For example, one may linearly interpolate $\left(m^{-1/2}\overline{X}(tm), 0\leq t\leq 1\right)$ at integer times. 

A key step in proving Theorem~\ref{thm:main} is to show that, when viewed from its barycenter, the process is Harris recurrent. 
\begin{dfn}\emph{(Harris recurrence, \cite{MT})}\label{dfn:HarrisRec}
We say that a time-homogeneous c\`adl\`ag Markov process $\Phi=(\Phi_t,t\geq 0)$ on the state space $\R^{d\times N}$ is Harris recurrent if there exists a $\sigma$-finite \jl{ and nonzero} Borel measure $\varphi$ on $\R^{d\times N}$ such that for any Borel set $A\subset \R^{d\times N}$ with $\varphi(A)>0$, for all $x\in \R^{d\times N}$,
\[
\bP_x(\eta_A=\infty)=1
\] where $\eta_A\coloneqq \int_0^\infty \mathbbm{1}_{\{\Phi_t\in A\}}dt$ is the total time spent in $A$ by $\Phi$.
\end{dfn}
Throughout the paper, we will denote particle configurations by $x=(x_i)_{i\in [N]}\in \R^{d\times N}$ where each $x_i\in \R^d$. Given a configuration $x\in \R^{d\times N}$, and another point $y \in \R^d$, we write $x-y := (x_i-y)_{i\in [N]}$. Equipped with this definition and notation, we state the second main result of the present work. 
\begin{thm}\label{thm:HarrisRec}
The process $\rX-\overline{\rX} := (X(t)-\overline{X}(t),t \ge 0)$ is Harris recurrent. 
\end{thm}

\jl{We prove Theorem~\ref{thm:HarrisRec} in Section \ref{subsec:HarrisRec}; for the proof we choose the measure $\varphi$ in the definition of Harris recurrence to be the $N$-fold product of the $d$-dimensional standard Gaussian measure.}

\subsection{Two constructions of barycentric Brownian bees, and a little additional notation}\label{sub:formalism}
We work on an abstract probability space $(\Omega, \cF, \bP)$, rich enough to support all of the random variables encountered. 

There are two distinct constructions of the barycentric Brownian bees (hereafter BBB) process that will be useful at different points in our analysis. 
The first construction closely parallels that given in Section 2 of \cite{ MR3784486}. 
\begin{itemize}
\item Let $\cJ=(\cJ_t,t\geq 0)$ be the counting process for a Poisson point process on $[0,\infty)$ with rate $N$, and for $i \ge 0$ let $J_i = \inf\{t\ge 0: \cJ_t=i\}$. The process $\cJ$ will be the number of branching events of $\rX$ up to time $t$, and $(J_i,i \ge 1)$ will be its branching times (and $J_0=0$). 
\item Let $(\cU_i;i\geq 1)$ be independent random variables, uniformly distributed on $\{1, \cdots, N\}$ and independent of $\cJ$. 
\item Let $(\rB_i)_{i\in [N]}$ be independent $d$-dimensional Brownian motions starting from the origin, which are independent of $\cJ$ and of $(\cU_i;i\geq 1)$. 
\end{itemize}
With these definitions at hand, the picture to have in mind when reading the formal construction is this: at time $J_i$, particle $\cU_i$ branches, and some particle is killed. If the killed particle has index $k$, then we may equivalently view this as the particle with index $k$ jumping to the location of the branching event, which is $X_{\cU_i}(J_i-)$. 

Formally, fix an initial configuration $x=(x_j)_{j \in [N]} \in \R^{d\times N}$. We may define the BBB inductively, as follows. Set $X_j(0)=x_j$ for each $j \in [N]$. Then, for each $i\geq 1$ and $j\in [N]$, let 
\[
X_j(t)\coloneqq X_j(J_{i-1})+B_j(t)-B_j(J_{i-1}), \;\;\; \mbox{ for }t\in [J_{i-1}, J_i).
\] This essentially states that in between jump times $J_{i-1}$ and $J_{i}$, the BBB process is nothing more than $N$ independent Brownian motions, appropriately shifted. Finally, for each $i \ge 1$, at time $J_i-$, we let particle $\cU_i$ branch, and kill (remove) the particle with index
\[
k\coloneqq \argmax_{j\in [N]} \left|X_j(J_i -)-\frac{1}{N+1}\left(X_{\cU_i}(J_i-)+\sum_{l\in [N]}X_l(J_i-)\right)\right|\, .
\] 
In words, we kill the particle which is furthest from the barycenter; the barycenter is calculated taking the newly born particle into account. We do not need to worry about ties since they occur with probability $0$.
This means that at time $J_i$, we set 
\begin{equation}\label{eq:firstconstruction}
X_j(J_i)=\begin{cases}X_{\cU_i}(J_i-)&j=k,\\X_j(J_i-)&j\in [N]\backslash\{k\}.\end{cases}
\end{equation}
This completes the first construction.

The second construction realizes the BBB process as embedded within a standard $d$-dimensional dyadic branching Brownian motion $W(t)=(W_i(t))_{i \in [\cN(t)]}$. The function $\cN:[0,\infty)\to\N$ counts the number of particles in the branching Brownian motion $\rW$ at any given time. Our convention is that when a branching event occurs in $\rW$, the new particle is appended to the end of the list of existing particles; if the particle with index $i$ branches at time $\tau$ then $\cN(\tau)=\cN(\tau-)+1$ and 
\[
W_j(\tau) = \begin{cases}
				W_j(\tau-)	&\mbox{ if }j < \cN(\tau),\\
				W_i(\tau-)	& \mbox{ if }j=\cN(\tau).
				\end{cases}
\]
The branching property means that after time $\tau$, the particles with indices $i$ and $\cN(\tau)$ evolve independently.

Given $x\in \R^{d\times N}$, we initialize $\rW$ with $N$ particles in $\R^d$ at positions $(x_i)_{i\in [N]}$, so that $\cN(0)=N$ and $W(0)=x$. To describe the embedding of $\rX$ within $\rW$, it suffices to explain the evolution of the vector $\cI(t)=(\cI_1(t),\ldots,\cI_N(t)) \in \cN(t)^{[N]}$ of indices of particles belonging to the BBB process at all times $t\geq 0$.

First, $\cI(0)=(1,2,\ldots,N)$, and $\cI$ is constant between branching events of $\rW$. If a particle with index $i \not\in \{\cI_1(\tau-),\ldots,\cI_N(\tau-)\}$ branches at time $\tau>0$ then $\cI(\tau)=\cI(\tau-)$. 
Finally, if for some $i \in [N]$ the particle with index $\cI_i(\tau-)$ branches at time $\tau$, then let 
\[
k = \argmax_{j \in [N]} \left|W_{\cI_j(\tau-)}(\tau)-\frac{1}{N+1}\left( W_{\cI_i(\tau-)}(\tau)+  \sum_{l \in [N]} W_{\cI_l(\tau-)}(\tau)\right) \right|\, .
\]
In other words, $k$ is the index of the particle furthest from the barycenter, when the barycenter is calculated taking the newly born particle (which has index $\cN(\tau)$ and position $W_{\cI_i(\tau-)}(\tau)=W_{\cN(\tau)}(\tau)$) into account. 

Now set 
\[
\cI_j(\tau)=
\begin{cases}
\cN(\tau)	& \mbox{ if } j = k\\
\cI_j(\tau-)	& \mbox{ if } j \ne k\, .
\end{cases}
\]

%

We can then realize the process $\rX$ as 
\begin{equation}\label{eq:secondconstruction}
\rX=(X_j(t))_{j\in [N]}\coloneqq (W_{\cI_j(t)}(t))_{j\in [N]}. 
\end{equation}
The realization of the BBB process within a branching Brownian motion will be useful in particular in Section \ref{sec:Constr}. 

\lou{
The second construction yields an easy description of the ancestral trajectories of BBB particles; if $\cI_j(t)=i$, so that $X_j(t)=W_i(t)$, then the 
ancestral trajectory of $X_j(t)$ is $(W_{i(s,t)}(s),0 \le s \le t)$, where $i(s,t)$ denotes the index of the time-$s$ ancestor of $W_i(t)$. Note that the ancestral trajectories are a.s.\ continuous. We use the notion of ancestral trajectories (and their continuity) with both BBB constructions at various points in the paper. }

We note that the labelling conventions of particles are consistent between the two constructions, in the sense that from a given configuration $(X_j(t-))_{j \in [N]}$, if there is a branching event at time $t$, then the definitions of the vectors $(X_j(t))_{t \in [N]}$ given by (\ref{eq:firstconstruction}) and by (\ref{eq:secondconstruction}) agree with one other.

At one point in the paper it will be useful to allow the BBB process to start from a configuration with fewer than $N$ particles. In this case no killing occurs until the population size reaches $N$; a branching event which occurs before this time simply increases the population size by one (with the convention that the newly-born particle is appended to the end of the list of existing particles, like in a BBM). For a BBB started from fewer than $N$ particles, we write $N(t)$ for the number of particles at time $t$, so $X(t)=(X_i(t))_{i \in [N(t)]}$ and a.s.\ $N(t)=N$ for all $t$ sufficiently large.

We write $(\cF_t,t \ge 0)$ for the natural filtration of $\rX$,
\[
\cF_t=\cF_t(\rX) \coloneqq \sigma((X_j(s))_{j\in [N]}, 0\leq s\leq t)\, .
\] 
\lou{When viewing $\rX$ as embedded within a BBM $\rW=(W(t),t \ge 0)$ via (\ref{eq:secondconstruction}), we will also need to work with the natural filtration of $\rW$, which we denote by $(\cG_t,t \ge 0)$. Note that in this case $\cF_t \subset \cG_t$, and $\rX$ is still a strong Markov process relative to the enlarged filtration.}

We write $\bP_x(\cdot)\coloneqq \bP(\cdot |X(0)=x)$ for $x\in \R^{d\times N}$. Likewise, for a Borel probability measure $\mu$ on $\R^{d\times N}$, we write $\bP_\mu$ for the probability measure under which the process starts from a $\mu$-distributed initial configuration, namely $\bP_\mu(X(0)\in A)=\mu(A)$ for Borel $A\subset \R^{d\times N}$. 

We use the notation $B(y,r)$ to denote the Euclidean ball centered at $y$ of radius $r$ in $\R^d$. Given two functions $f,g:\R\to \R$ with $g\not = 0$, we write $f\overset{m\to\infty}{\sim}g$ if $\frac{f(m)}{g(m)}=1+o(1)$ as $m\to\infty$. Finally, for $\R$-valued random variables $Z_1$ and $Z_2$, we say that $Z_1$ stochastically dominates $Z_2$ (with respect to a given probability measure $\bP$), if 
\[
\bP(Z_1\geq y)\geq \bP(Z_2\geq y), \;\;\; \forall y\in \R\,. 
\] 

\subsection{Overview of the proof}\label{s.pfOverview}
In this section, we first describe a general mechanism for obtaining invariance principles for c\`adl\`ag processes. We then sketch how it is applied in the current setting, and in particular how Harris recurrence comes into play, and conclude with an overview of the proof that, when viewed from its barycenter, the BBB process is Harris recurent. 

Fix $N,d\in \N$. Given an arbitrary $\cD([0,\infty), \R^{d\times N})$-valued stochastic process $ \rX=((X_j(t))_{j\in [N]}, t\geq 0)$, we define the barycenter $\overline{X}(t)\coloneqq N^{-1}\sum_{1\leq i\leq N}X_i(t)$; so $\overline{\rX}=(\overline{X}(t),t\geq 0)$ is an $\R^d$-valued stochastic process. We say that $\overline{\rX}$ satisfies an invariance principle if there exist $\alpha \in \R^d$ and a $d\times d$ matrix $\Sigma=\Sigma(d,N)$ such that 
\begin{equation}\label{e.geninv}
\left(m^{-1/2}\left(\overline{X}(tm)-tm\alpha\right), 0\leq t\leq 1\right)\xrightarrow{d}\left(\Sigma B(t), 0\leq t\leq 1\right) 
\end{equation} with respect to the Skorohod topology on $\cD([0,1], \R^{d\times N})$, where $\rB=(B(t), 0\leq t\leq 1)$ is a standard $d$-dimensional Brownian motion. 
The following general proposition identifies sufficient conditions for $\overline{\rX}$ to satisfy an invariance principle. 

\begin{prop}\label{prop:prop2}
Let $(\Omega,\cF,\bP)$ denote a probability space and $\rX:\Omega \to \cD([0,\infty), \R^{d\times N})$ a c\`adl\`ag stochastic process with $\rX=((X_j(t))_{j\in [N]}, t\geq 0)$ with $X(0)=x$ for some $x\in \R^{d\times N}$.  Suppose that there exists a nondecreasing sequence $(\tau_i; i\geq 1)$ of $\bP_x$-a.s. nonnegative, finite random variables such that the following properties are satisfied.
\begin{enumerate}
 \item (IID time increments) The time increments
    \[
    (\tau_{i+1}-\tau_i; i\geq 1)
     \] are independent and identically distributed with finite and \jl{nonzero} mean.
    \item (IID increments) The $\R^{d}$-valued random variables 
    \[
    \left( (X_1(\tau_{i+1})-X_1(\tau_i));i\geq 1\right)
    \] are independent and identically distributed.
    \item (Identical distribution between times) The $\R$-valued random variables 
    \[
    \left(\sup_{\tau_i\leq t\leq \tau_{i+1}}|X_1(t)-X_1(\tau_i)|; i\geq 1\right)
    \]
    are identically distributed.
    \item (Finite variance increments) We have 
    \[
    \bE_x\left[\sup_{\tau_1\leq t\leq \tau_2}|X_1(t)-X_1(\tau_1)|^2\right]<\infty.
    \]
    \item (Barycenter approximation in probability) We have
    \[
    m^{-1/2}\sup_{0\leq t\leq 1}|\overline{X}(tm)-X_1(tm)|\to 0
    \] in probability as $m\to\infty$.
\end{enumerate} 
    Then the barycenter $\overline{\rX}$ of $\rX$ satisfies an invariance principle \eqref{e.geninv}. Moreover, the rescaling $\Sigma$ and the drift $\alpha$ are given by 
    \[
    \Sigma= \bE_x[\tau_2-\tau_1]^{-1/2}Q
    \]
     where $Q$ is a $d$-by-$d$ matrix such that $C\coloneqq QQ^T$ is the covariance matrix of $X_1(\tau_2)-X_1(\tau_1)$, and 
     \[
    \alpha= \bE_x[X_1(\tau_2)-X_1(\tau_1)]\bE_x[\tau_2-\tau_1]^{-1}.
    \]
\end{prop}  

The proof of Proposition~\ref{prop:prop2} appears in the \hyperref[sec:append]{Appendix}. To apply the proposition in our setting, we begin by constructing a sequence of ``regeneration'' times for the BBB process~$\rX$ with properties $(1),(2)$ and $(3)$ of Proposition \ref{prop:prop2} in mind. The prototypical regeneration time will be defined as the first time $\rho>0$ when two specific events have occurred. The events are constructed in such a way that when they both occur, the process starts over (regenerates) from a single \textit{queen particle}, by which we mean that all particles in the BBB at time $\rho$ are descendants of $X_{1}(\rho-1)$. At a time $\rho$ when these two events have just taken place, the law of $(X_j(\rho)-X_{1}(\rho-1))_{j\in [N]}$ is given by an explicit measure and is independent of the state of the process at any previous regeneration times $\rho'\leq \rho$. 

In order to verify the conditions of Proposition~\ref{prop:prop2}, the key will be to control the amount of time required for regeneration from a single queen particle to occur. 
The events required for regeneration are defined in Section~\ref{subsec:ConstrOutline}; for now we simply note that they shall involve (a) a specific branching pattern occurring over a bounded time interval, and  (b) the particles staying in certain bounded regions over the same time interval. The probability that these events occur will be easy to bound from below, provided that all particles start at a bounded distance from their barycenter.  

The latter requirement is closely connected to the Harris recurrence of $\rX-\overline{\rX}$. Harris recurrence essentially requires that a process eventually spends an infinite amount of time in any sufficiently large set, where ``large''  is measured by a $\sigma$-finite Borel measure which we are free to choose. We prove this for $\rX-\overline{\rX}$ by showing that the time required for $\rX$ to reach a configuration in which all particles are a  bounded distance from the barycenter has exponential tails (see Remark \ref{r.HarrisRec} (ii) below), so in particular has finite moments of all orders.
This will easily yield Harris recurrence. It will also allow us to deduce that our renewal time increments have exponential tails, and thence to 
show that the renewal times satisfy the conditions 
of Proposition \ref{prop:prop2}. Theorem \ref{thm:main} will then straightforwardly follow from Proposition~\ref{prop:prop2}.

The remainder of this section provides an overview of the proof that $\Phi\coloneqq \rX-\overline{\rX}$ is Harris recurrent. In broad strokes, the proof proceeds as follows. Given given points $x_1,\ldots,x_N$ in $\R^d$,  say that $x=(x_1,\ldots,x_N)$ is {\em unambiguous} if for any vector $(f_1,\ldots,f_N)$ of non-negative integers with $\sum_{i=1}^N f_i=N+1$, 
\begin{equation}\label{e.defunam}
\Big|x_j-\frac{1}{N+1}\sum_{i=1}^N f_ix_i\Big| \ne \Big|x_k-\frac{1}{N+1}\sum_{i=1}^N f_ix_i\Big|
\end{equation}
for all $1 \le j < k \le N$. Here is the meaning of this condition. 
Suppose that at some point in the BBB process, the particles have  positions $x_1,\ldots,x_N$. Now imagine that a sequence of branching events occurs in rapid succession, rapidly enough that the particles essentially do not move during the course of the branching. If at some point during such a sequence, the moment a branching event occurs, there are $f_i$ particles at (or extremely near to) position $x_i$ (for each $1 \le i \le N$), then the barycenter of the configuration is at (or extremely close to)
\[
\frac{1}{N+1}\sum_{i=1}^N f_ix_i. 
\]
If $x$ is unambiguous, then at this point there is a unique particle which is furthest from the barycenter, so there is no ambiguity about which particle to kill. 

We actually use a more quantitative version of the above definition. For $\delta >0$, say that $x=(x_1,\ldots,x_N)$ is $\delta$-unambiguous if 
for any vector $(f_1,\ldots,f_N)$ of non-negative integers with $\sum_{i=1}^N f_i=N+1$, for all $1 \le j < k \le n$, we have 
\begin{equation}\label{e.defdelta}
\left|\Big|x_j-\frac{1}{N+1}\sum_{i=1}^N f_ix_i\big| - \big|x_k-\frac{1}{N+1}\sum_{i=1}^N f_ix_i\Big|\right| > \delta. 
\end{equation}
It is not hard to show that if $x=(x_1,\ldots,x_N)$ is $\delta$-unambiguous and $y=(y_1,\ldots,y_N)$ is such that $|x_i-y_i| <\delta/4$ for $1 \le i \le N$, then for any non-negative integers $(f_1,\ldots,f_N)$ with $\sum_{i=1}^N f_i=N+1$, the index $j$ for which 
$\left|x_j-(N+1)^{-1}\sum_{i=1}^N f_ix_i\right|$ is maximized, also maximizes $\left|y_j-(N+1)^{-1}\sum_{i=1}^N f_iy_i\right|$. This means that in the BBB process, if particles stay within distance $\delta/4$ of a fixed $\delta$-unambiguous configuration $\rx$ for a given period of time, then during that time,  for the purpose of identifying which particle is furthest from the barycenter, one may pretend that all particles within distance $\delta/4$ of point $x_i$ are in fact colocated at $x_i$. In other words, during that time the effect of spatial motion can be ignored when calculating which particles are killed after branching events; we may just pretend all particles are located at one of $x_1,\ldots,x_N$. 

The following straightforward lemma implies that from any initial configuration, the BBB process quickly reaches a $\delta$-unambiguous configuration for some $\delta > 0$. Write 
\[
\xi_\delta := \inf\{t \ge 0: X(t)\mbox{ is $\delta$-unambiguous}\}\, .
\]
\begin{lem}\label{lem:unambiguous_fast}
For any $\eps>0$, there exists $\delta=\delta_\eps>0$ such that 
\[
\inf_{x\in \R^{d\times N}}\bP_x(\xi_\delta<1)>1-\epsilon.
\]
\end{lem}

The value of reducing to the unambiguous setting is that it allows us to study sequences of branching and killing events for deterministic configurations in order to make deductions about the possible behaviour  of the BBB process. An {\em unambiguous configuration} consists of an unambiguous vector $x=(x_1,\ldots,x_N)$ and a vector $w=(w_1,\ldots,w_N)$ of non-negative integer weights with $\sum_{i=1}^{N}w_{i}=N$. 
A branching event for site $\ell\in [N]$ corresponds to increasing $w_\ell$ by $1$, and a killing event for particle $k\in [N]$ corresponds to decreasing $w_k$ by $1$. Formally, given $x,w$ as above, for $\ell \in [N]$ let 
\begin{equation}\label{e.selection}
k = k(x,w,\ell)=\argmax_{j\in [N], w_j>0}\left|x_{j}-\frac{1}{N+1}\Big(\sum_{1 \le i \le N}w_ix_i + x_\ell \Big)\right|
\end{equation}
and define weights $g(x,w,\ell)=(g_1,\ldots,g_N)$ by 
\[
g_j = w_j + \mathbbm{1}_{\{j=\ell\}} - \mathbbm{1}_{\{j=k\}}\, . 
\]
Note that $\sum_{i=1}^N g_i = N$, so the pair $x, g(x,w,\ell)$ again form an unambiguous configuration. 

Given an unambiguous configuration $x,w$, any sequence $l_1,\ldots,l_m$ of elements of  $[N]$ induces a sequence of weight vectors $w^{(0)},w^{(1)},\ldots,w^{(m)}$, by letting particle $l_i$ branch at step $i$. Formally, set $w^{(0)}=w$, and for $1 \le i \le m$ let $w^{(i)}=g(x,w^{(i-1)},l_i)$. We say $l_1,\ldots,l_m$ is a valid sequence for $x,w$ if $w_{l_i}^{(i-1)}>0$ for all $1 \le i \le m$, i.e. if branching only occurs for particles with positive weight.  
\begin{lem}\label{lem:regen_unambiguous}
Given any unambiguous configuration $x,w$, there exists an integer $m \le (N-1)^2$ and a sequence $l_1,\ldots,l_m \in [N]$ which is valid for $x,w$ such that $w^{(m)}$ contains exactly one non-zero entry. 
\end{lem}
We say that a sequence $l_1,\ldots,l_m$ as in Lemma~\ref{lem:regen_unambiguous} {\em collapses} the (unambiguous) configuration $x,w$. The proof of Lemma~\ref{lem:regen_unambiguous}, which is the key step in the proof of Harris recurrence, is found in Section~\ref{subsec:deterministic}, below.


To prove Harris recurrence for $\Phi=\rX-\overline{\rX}$, we now argue as follows. Step 1, we wait until the first time $t$ that $X(t)$ is $\delta$-unambiguous (for some small fixed $\delta > 0$); by Lemma~\ref{lem:unambiguous_fast} and the strong Markov property this takes a bounded amount of time. Step 2, we hope that in the time interval $[t,t+1]$, (a) all particles stay within distance $\delta/4$ of their time-$t$ locations, and (b) a sequence of branching events occurs which collapses the configuration $X(t),(1,1,\ldots,1)$. (The all-ones vector represents the fact that at time $t$ there is one particle at each location.) If Step 2 succeeds, then there is $j \in [N]$ such that in $X(t+1)$ all particles lie within distance $\delta/4$ of $X_j(t)$, so all entries of $\Phi(t+1)$ have distance at most $\delta$ from the origin. 

By Lemma~\ref{lem:regen_unambiguous}, there is at least one  sequence consisting of at most $(N-1)^2$ branching events which collapses the configuration, from which it is not hard to see that Step 2(b) succeeds with probability bounded away from zero uniformly over the possible values of $X(t)$. If Step 2 fails, we start over from Step 1. The Markov property then guarantees that both steps will succeed after a \jl{finite} number of trials. This means that from any initial configuration, $\Phi$ will reach a state where all particles are at bounded distance from the origin in \jl{finite} time. Combined with the fact that between branching events, the process $\rX$ behaves like $N$ independent Brownian motions, this allows us to prove Harris recurrence by comparing the law of the process $\Phi$ at a fixed time to the joint law of $N$ independent $d$-dimensional standard Gaussians. 
 The full details of this argument appear in Section~\ref{subsec:HarrisRec}.

\subsection{Related work}\label{sub:related}
The BBB model may be seen as a close relative of the Brunet-Derrida particle systems \cite{MR2299937,MR2365627} and their multidimensional generalizations \cite{MR3784486}. A Brunet-Derrida particle system also consists of a population of a fixed size $N$ undergoing both branching and spatial displacement, with individual trajectories taking values in $\R$. The law of the system is determined by a {\em fitness function} $s:\R \to \R$; when a branching event occurs, particles of lowest fitness are removed from the system to keep the population size constant. (The original papers on the subject \cite{MR2299937,MR2365627} considered discrete-time processes -- branching random walks with selection, in which multiple branching events can occur simultaneously. Much of the subsequent work in the area has focused on branching Brownian motion, as in this paper.) The asymptotic displacements \cite{MR2669438,MR3568046}, evolution of empirical particle densities \cite{MR2932664,MR3986078}, and typical genealogies \cite{MR3077519} of such systems have all received attention in the literature. The paper \cite{MR3327315} considers a related model with branching and barycentric elimination but no particle movement. 

In \cite{MR3784486}, Berestycki and Zhao consider branching Brownian motion with selection in $\R^d$. The model is again specified by a fitness function, $s:\R^d \to \R$; a fixed number $N$ of particles move as independent Brownian motions which independently branch at rate 1. On a branching event, a particle of minimal fitness is removed from the system. Berestycki and Zhao prove, among other results, the following facts. When $s(x)=|x|$, \jl{in the large-time limit,} the particle system asymptotically travels at a constant speed, in a uniformly random direction, in a cloud of particles whose diameter is $O(1)$ in probability. 
When the fitness function is linear, $s(x)=\langle x,v \rangle$ for a fixed $v \in \mathbb{S}^{d-1}$, the cloud of particles asymptotically travels with constant speed in direction $v$. In this case (under some conditions on the time-zero configuration of the particles), the diameter of the cloud of particles in direction $v$ is asymptotically $O(\log N)$ in probability, whereas for any $w \in \mathbb{S}^{d-1}$ with $\langle v,w\rangle=0$, the diameter of the cloud in direction $w$ is at least $O(\log N)^{3/2}$ in probability. 

We also mention some recent and forthcoming works on branching systems with selection, with some similarity to ours. 
\begin{itemize}
\item Berestycki, Brunet, Nolen and Penington \cite{2020berestycki, 2020berestycki2} study a branching Brownian motion \\$((X_i(t))_{i\in [N]}, t\geq 0)$ in $\R^d$ with $N$ particles, with a selection rule that can be specified via the fitness function $s(x)=-|x|$; in other words, on a branching event the particle farthest from the origin is removed. For such a model, \cite{2020berestycki, 2020berestycki2} proves results on the long-time and large-particle behaviour of the empirical distribution of particles in the system, by developing a connection with certain free-boundary problems. We discuss these works further in the conclusion, Section \ref{sec:conclusion}.
\item The works \cite{beckman19thesis,beckman20hydro} consider binary branching Brownian motion $(X_i(t))_{i \in [\mathcal{N}(t)]}$ in $\R^d$ with a variable number of particles, where the branching/death rate at position $x$ at time $t$ has the form 
\[
\mu_t(x) = g\left(\int_\R f(x-y)\nu_t(\mathrm{d}y)\right)\, ; 
\]
here $f,g:\R \to \R$ are Borel functions, and $\nu_t=\mathcal{N}(t)^{-1} \sum_{i=1}^{\mathcal{N}(t)} \delta_{X_i(t)}$ is the empirical measure of the time-$t$ particle distribution. When $\mu_t(x)$ is positive it is interpreted as a branching rate; when this quantity is negative it is interpreted as a death rate. For certain choices of the functions $f$ and $g$, Beckman \cite{beckman19thesis} proves hydrodynamic limit theorems for such systems, showing that the evolution of the empirical density is governed by certain integro-differential equations; other models of the same sort are considered in \cite{beckman20hydro}. 
\end{itemize}

\subsection{Outline of the paper}
The rest of the paper is organized as follows. Theorem \ref{thm:main} is a consequence of Proposition \ref{prop:prop2} for an appropriately defined sequence of times $(\tau_i; i\geq 1)$. In Section \ref{sec:Constr}, we construct such a sequence $(\tau_i; i\geq 1)$ for the BBB process and state several key lemmas about these times, which we then use to demonstrate they satisfy the hypotheses of Proposition \ref{prop:prop2}. We conclude Section \ref{sec:Constr} with the proof of Theorem \ref{thm:main}. In Section \ref{sec:properties}, we prove the technical lemmas used in the construction of the times in Section \ref{sec:Constr}, and also prove Harris Recurrence (Theorem \ref{thm:HarrisRec}), which follows relatively easily from one of these lemmas. Section \ref{sec:conclusion} contains concluding remarks, {\color{cjl} open questions about the process, and suggestions for future research}. Finally, in the Appendix (Section \ref{sec:append}), we present the proof of Proposition \ref{prop:prop2}.

\section{Construction of $(\tau_{i}; i\geq 1)$ and the Proof of Theorem \ref{thm:main}}\label{sec:Constr}
Throughout this subsection, unless otherwise \jl{stated}, we assume that $N\geq 3$. The first portion of this section is devoted to the construction of $(\tau_i; i\geq 1)$, and to introducing two technical lemmas which motivate the construction. We thereafter show that $(\tau_{i}; i\geq 1)$ satisfies properties (1)-(5) of Proposition \ref{prop:prop2}. We conclude with the proof of Theorem \ref{thm:main}. 

In fact, we will focus on constructing a sequence of regeneration times $(\rho_{i}; i\geq 1)$ for which the law of the BBB process has an explicit and fairly simple form (and, after recentering, is the same for all $i\geq 1$). We will then define $(\tau_{i}; i\geq 1)$ from these $(\rho_{i}; i\geq 1)$ (by letting $\tau_{i}:=\rho_{i}-1$), for which we will be able to verify properties (1)-(5) of Proposition \ref{prop:prop2}.

\lou{In the coming section, we work with the construction the BBB process as embedded within a BBM $\rW=(W(t),t \ge 0)$, and recall from (\ref{eq:secondconstruction}) that, with this construction, we have $(X_j(t))_{j\in [N]}=(W_{\mathcal{I}_j(t)}(t))_{j\in [N]}$.}

\subsection{Construction of $(\rho_i; i\geq 1)$ and some key ideas}\label{subsec:ConstrOutline}
For $I\subseteq \mathbb{N}$, given a collection $x = (x_i)_{i \in I}$ of vectors in $\R^d$, the {\em extent} of $x$ is defined to be
\[
E(x) := \sup_{i,j \in I} |x_j-x_i|.
\]
For $t > 0$, we define an event $A_t$, measurable with respect to \lou{$\sigma(X(s),t \le s \le t+1)\subset\sigma(W(s),t \le s \le t+1)$}, as follows. Write $G'\coloneqq\{2,\ldots,\left\lceil \frac{N+1}{2}\right\rceil\}$ and $D'\coloneqq\{\left\lceil \frac{N+1}{2}\right\rceil+1,\ldots,N\}$, and note that $|G'|\geq |D'|$. Let $\rad:=\frac{1}{4(N+1)}$. The event $A_t$ occurs if and only if the following events all take place.

\begin{list}{ (\theatscan)}
{
\usecounter{atscan}
\setlength{\topsep}{1.5ex plus 0.2ex minus 0.2ex}
\setlength{\labelwidth}{1.2cm}
\setlength{\leftmargin}{1.5cm}
\setlength{\labelsep}{0.3cm}
\setlength{\rightmargin}{0.5cm}
\setlength{\parsep}{0.5ex plus 0.2ex minus 0.1ex}
\setlength{\itemsep}{0ex plus 0.2ex}
}
\item \label{e.at1} There are no branching events in $\rX$ in the time interval $[t,t+1]$. 
\item \label{e.at2} $X_j(t+1)-\overline{X}(t)\in B\left(-5e_1,\rad\right)$ for all $j\in G'$, and $X_j(t+1)-\overline{X}(t)\in B\left(5e_1, \rad\right)$ for all $j\in D'$.
\item \label{e.at3} $X_1(t+1)-\overline{X}(t)\in B\left(\gamma, \rad\right)$, where 
\begin{equation*}\label{e.gadef}
\gamma\coloneqq \frac{1}{N-1}\left(-5\left\lceil\frac{N-1}{2}\right\rceil+5\left\lfloor\frac{N-1}{2}\right\rfloor\right)e_{1}=\begin{cases}
0&\text{$N$ odd,}\\
-\frac{5}{N-1}e_{1}&\text{$N$ even}
\end{cases}
\end{equation*} is roughly the position of the time-$(t+1)$ barycenter (reshifted by $\overline{X}(t)$) of the particles in $G'\cup D'$ along the direction of the standard unit vector $e_1=(1,0,\cdots, 0)^T$. 
\end{list}

In words, $A_{t}=A_{t,1}\cap A_{t,2}\cap A_{t,3}$ is the event that during the time interval $[t,t+1]$, no branchings occur and no single trajectory makes a large displacement; particle $X_{1}$ moves to a position close to the time-$(t+1)$ barycenter; and the remaining particles cluster into two roughly equal groups, both with small extent and both reasonably far from the barycenter and from each other. The shift by $\overline{X}(t)$ is simply an adjustment to make the event easier to describe (by ensuring that the starting particle configuration has barycenter located at 0). Moreover, the three clusters of particles with indices in $G'$, $D'$, and $\{1\}$ are approximately collinear. 

{\color{cjl}
Next, for $t\geq 0$, we define an event $B_t$, measurable with respect to $\sigma(W(s),t+1 \le s \le t+2)$. Before proceeding to details, we note that in this definition the embedding of the BBB within the BBM is crucial, as we shall impose constraints on particle trajectories after the particles have already died within the BBB process. The reason for such constraints is that they allow us to gain a degree of independence, which is necessary for establishing a renewal structure using the regeneration times.

The event $B_t$ occurs if and only if the following events both take place. 
\begin{list}{ (\thebtscan)}
{
\usecounter{btscan}
\setlength{\topsep}{1.5ex plus 0.2ex minus 0.2ex}
\setlength{\labelwidth}{1.2cm}
\setlength{\leftmargin}{1.5cm}
\setlength{\labelsep}{0.3cm}
\setlength{\rightmargin}{0.5cm}
\setlength{\parsep}{0.5ex plus 0.2ex minus 0.1ex}
\setlength{\itemsep}{0ex plus 0.2ex}
}

    \item \label{e.bt1} In the time interval $[t+1,t+2]$, particle $X_1(t+1)=W_{\cI_1(t+1)}(t+1)$ and its descendants branch at least $N-1$ times. Moreover, for $i \in \{2,3,\ldots,N\}$, the particle $W_{\cI_i(t+1)}(t+1)$ does not branch in the time interval $[t+1,t+2]$. 
        \item \label{e.bt2} For all $j \in [N]$, all $s \in [t+1,t+2]$ and all $k \in \cN(s)$, if $W_k(s)$ is a descendant of $W_{\cI_j(t+1)}(t+1)$ then $|W_k(s)-W_{\cI_j(t+1)}(t+1)| \le \rad$.
\end{list}

The event $B_{t,1}$ in particular requires that in the BBB process, only $X_1(t+1)$ and its descendants branch in the time interval $[t+1,t+2]$; but it also requires that there is no branching among BBM descendants of the other BBB individuals on this time interval, even if those individuals have died within the BBB process. Likewise, the event $B_{t,2}$ requires that on the time interval $[t+1,t+2]$, every BBM particle remains at distance less than or equal to $\rad$ from the location of its time-$(t+1)$ ancestor, if that ancestor was a BBB particle at time $t+1$. It therefore can constrain the motion of BBM descendants of BBB particles even after their deaths within the BBB process.
}

Although the event $B_{t}$ only requires certain branching events to occur and minimal movement of the particles, we now show that the event $A_{t}\cap B_{t}$ forces all particles besides $X_{1}$ and its descendants to be killed in the time interval $[t+1, t+2]$.

We introduce $\cS$, the set of all possible configurations $x\in \R^{d\times N}$ such that there exist disjoint sets $G, C$ and $D$ with $C\not = \emptyset$, and $G\cup C\cup D=[N]$ with 
\begin{enumerate}
\item[(a)] $x_i\in B(\gamma, 2\rad)$ for all $i\in C$, where $\gamma$ was defined by \eqref{e.at3}.
\item[(b)] $x_j\in B(-5e_1, 2\rad)$ for all $j\in G$, and $x_j\in B(5e_1, 2\rad)$ for all $j\in D$.
\item[(c)] $|G\cup C|\geq  \left\lceil\frac{N-1}{2}\right\rceil+1$ and $|D\cup C|\geq \left\lfloor\frac{N-1}{2}\right\rfloor+1$.
\end{enumerate}
Conditioned on the event $A_{t}\cap B_{t}$, we note that $X(s)-\overline{X}(t)\in \mathcal{S}$ for all $s\in [t+1, t+2]$. Therefore, to determine which particles of $\rX$ get killed in time $[t+1, t+2]$, it suffices to show that for all configurations belonging to $\mathcal{S}$, a branching event in $C$ implies that, under the same killing rules as BBB (namely that the particle farthest away from the barycenter is killed), the killing event occurs in $G\sqcup D$. 



\begin{figure}
  \centering
  \includegraphics[width=\linewidth]{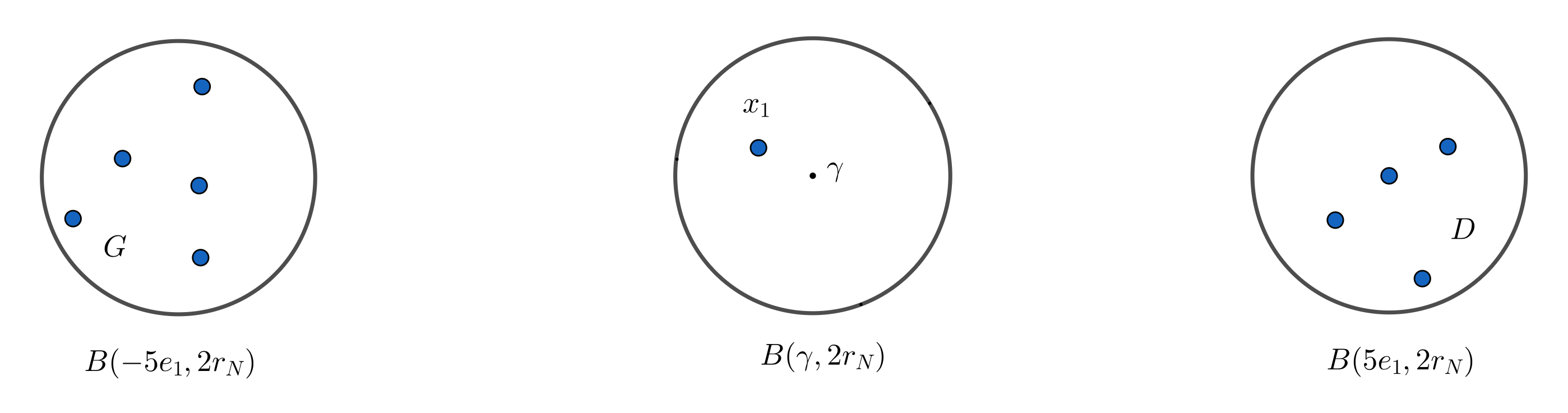}
  \caption{Example of particle configuration $x=(x_i)_{i\in[N]}$ in $\cS$ with $C=\{1\}$ which could arise just after event $A$ has occurred. The particles in the ball on the left are $(x_i)_{i\in G}$, and those in the ball on the right are $(x_i)_{i\in D}$.}
  \label{fig:figInS} 
\end{figure} 

\

\begin{lem}\label{l.no_kill_center}
For any configuration $x\in \cS$ \jl{with $G\sqcup D\neq \emptyset$,} , a branching event in $C$ implies that the killed particle belongs to $G\sqcup D$. 
\end{lem}
\begin{proof}
Suppose first that \jl{both} $G, D\not= \emptyset$ and that the particle with index $\ell\in C$ branches. Let $\overline{x}^{N+1}$ denote the barycenter of the system with $N+1$ particles (before a killing event occurs). Then $\overline{x}^{N+1}$ lies in the convex hull of the particles $\{x_j\}_{j\in [N]}\cup \{x_{\ell}\}$. Since the clusters of particles $G$, $C$, and $D$ are roughly collinear with $C$ located in between $G$ and $D$, it follows that the resulting killing occurs in either $G$ or $D$. 

We may therefore assume without loss of generality that $D=\emptyset$, $|C|\geq \left\lfloor\frac{N-1}{2}\right\rfloor+1$, and $|G|\leq \left\lceil \frac{N-1}{2}\right\rceil$. Upon a branching event in $C$, we observe that for the system with $N+1$ particles (before a killing occurs), $|C|>|G|$.  A small computation then shows that $\overline{x}^{(N+1)}$ lies in the halfspace $\left\{y\in \R^d:y\cdot e_1\geq \iota - 2\rad \right\}$, where 
\[
\iota\coloneqq \frac{1}{N+1}\left[-5\left\lceil\frac{N-1}{2}\right\rceil-\left(\left\lfloor\frac{N-1}{2}\right\rfloor+2\right)\gamma\right]=\begin{cases}
-\frac{5}{2}\frac{N-1}{N+1}&\text{$N\geq3$ odd,}\\
-\frac{5}{2}\frac{N^2+2}{N^2-1}&\text{$N\geq3$ even.}
\end{cases}
\]
The point $\iota e_{1}$ roughly corresponds to the barycenter of the $(N+1)$ system of particles after a branching event in $C$. Using now that for all $i\in G$, $x_{i}\in B\left(-5e_{1}, 2\rad\right)$ and for all $j\in C$, $x_{j}\in B\left(\gamma, 2\rad\right)$, still with $\rad=\frac{1}{4(N+1)}$, we see that 
\begin{equation*}
\min_{i\in G'} |x_{i}-\overline{x}^{(N+1)}|> \max_{j\in C'}|x_{j}-\overline{x}^{(N+1)}|.
\end{equation*}
Therefore the particle being killed must belong to $G$.

\end{proof}

\begin{remark}\label{r.queen}
Conditioned on $A_{t}\cap B_{t}$, since $X(s)-\overline{X}(t)\in \mathcal{S}$ for all $s\in [t+1, t+2]$, it follows by Lemma \ref{l.no_kill_center} that all particles with indices \jl{in $G\sqcup D$ die}, so at time $t+2$ all particles are descendants of $X_1(t+1)$. The resulting particle configuration looks like the one shown in Figure~\ref{fig:figAfterB}.
\end{remark}
\begin{remark}
The proof of Lemma~\ref{l.no_kill_center} can be seen as a simple case of using ``deterministic configurations'' to deduce which killing events occur, when the particles are allowed very limited movement. A more complex formulation of similar ideas appears in the proof of Lemma~\ref{lem:lem10}.
\end{remark}


\begin{figure}
  \centering
  \includegraphics[width=\linewidth]{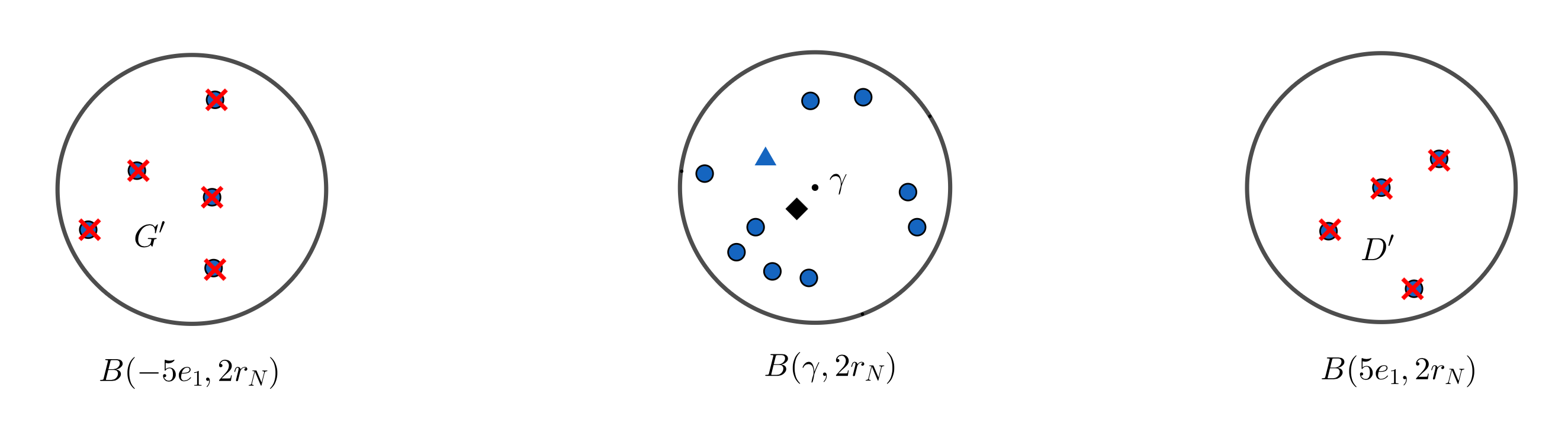}
  \caption{Example of a time-$(t+2)$ configuration after $A_t\cap B_t$ has occurred. The crossed-out particles indicate that when $A_t\cap B_t$ occurs, at time $t+2$, the particles in $(X_i(t+1))_{i\in G'\cup D'}$ live in $B(\gamma, 2\rad)$ (by Lemma~\ref{l.no_kill_center}). The blue triangle is the location of $X_1(t+2)-\overline{X}(t)$, and the black square is the location of $\overline{X}(t+2)-\overline{X}(t)$. }
  \label{fig:figAfterB} 
\end{figure}

Observe that the events $A_t$ and $B_t$ are invariant under translations in $\R^d$. Namely, $A_t$ and $B_t$ occur for 
\lou{$\rW$} if and only if they occur for 
\lou{$\rW+y\coloneqq (W(t)+y,t \ge 0)$} for every $y\in \R^d$.
Previously $\rW$ and $\rW+y$ were $\rX$ and $\rX+y$.

For any $L>0$, define 
\[
C_L\coloneqq\{x\in \R^{d\times N}: E(x)\leq L\},
\] the set of deterministic $N$-particle configurations $(x_i)_{i\in [N]}$ with extent at most $L$. Denote $A\coloneqq A_0$ and $B\coloneqq B_0$. It turns out that as long as we start from configurations of bounded extent, there is a uniformly positive probability that $A\cap B$ occurs. 

\begin{lem} \label{lem:lem6}
For all $L>0$, there exists $\epsilon=\epsilon_{L,N}>0$ such that \[\inf_{x\in C_{L}}\bP_{x}(A\cap B)>\epsilon.\]
\end{lem}

Since the lower bound in Lemma~\ref{lem:lem6} depends on $L$, we will want to show that regardless of the extent of the starting configuration, the BBB reaches a state with extent $L$ in finite time. This is a consequence of the next lemma.

\begin{lem}\label{lem:lem10} \jl{For $T^{(L)}\coloneqq \inf\{\jl{t\geq 1}: E(X(t))\leq L\}$, we have \[\lim_{L\to\infty}\sup_{x\in \R^{d\times N}}\bP_x\left(T^{(L)}>L\right)=0.\]}
\end{lem} 
\begin{remark}\label{r.HarrisRec}
\begin{enumerate}
\item[(i)] Lemma~\ref{lem:lem10} holds for all $N\geq 1$. The cases $N=1,2$ are trivial, and the case $N\geq 3$ will be proved in Section \ref{subsec:HarrisRec}.
\item[(ii)] It is a direct consequence of Lemma~\ref{lem:lem10} that $T^{(L)}$, the time to reach a configuration with extent at most $L$, has exponential tails. Indeed, it suffices to choose $L>0$ such that $\sup_{x\in \R^{d\times N}}\bP_x(T^{(L)}>L)<\frac{1}{2}$ and then apply the Markov property.
\end{enumerate}
\end{remark}
We will see that Lemma~\ref{lem:lem10} implies Harris recurrence (Theorem \ref{thm:HarrisRec}). Equipped with Lemmas \ref{lem:lem6} and \ref{lem:lem10}, we can now construct the stopping times $(\rho_i;i\geq 1)$ as follows. Fix $L \ge 2\rad$ such that $\sup_{x\in \R^{d\times N}}\bP_x(T^{(L)}>L)<\frac{1}{2}$, and set $T_1\coloneqq T^{(L)}=\inf\{t \ge 1: E(X(t)) \le L\}$. For each $i\geq 1$, recursively define the stopping times
\begin{equation}\label{eq:ti_def}
T_{i+1}\coloneqq \inf\{t\geq T_i+2:E(X(t))\leq L\},\;\;\;i\geq 1,
\end{equation} 
and let
\[
I_1\coloneqq\inf\{i\geq 1:A_{T_i}\cap B_{T_i}\textnormal{ occurs}\}. 
\] 
By Remark \ref{r.HarrisRec} (ii), and the strong Markov property, $T_i<\infty$ almost surely for each $i\geq 1$. By the strong Markov property and Lemma~\ref{lem:lem6}, $A_{T_i}\cap B_{T_i}$ has positive probability uniformly over all possible initial configurations for each $i\geq 1$. In particular, an application of the second Borel-Cantelli lemma implies that $I_1<\infty$ $\bP_x$-almost surely. Thus, we can define $\tau_1\coloneqq T_{I_1}+1$ and $\rho_1\coloneqq T_{I_1}+2$; both $\tau_1$ and $\rho_1$ are $\bP_{x}$-a.s. finite.  By construction, we observe that at time $\rho_1$, the events $A_{T_{I_1}}$ and $B_{T_{I_1}}$ have occurred (as depicted in Figure \ref{fig:figAfterB}), and at time $\rho_1-2$, the particles had extent at most $L$. 

Next, for each integer $l> 1$, inductively set
\[
I_{l}\coloneqq \inf\{i>I_{l-1}:A_{T_i}\cap B_{T_i}\text{ occurs}\}, 
\]
and let 
\begin{align}
\tau_l& \coloneqq T_{I_{l}}+1\, ,\mbox{ and } \label{eq:tau_def}\\
\rho_l& \coloneqq T_{I_{l}}+2. \notag
\end{align}
At time $\tau_l$, the event $A_{T_{I_{l}}}$ has just occurred, and the event $B_{T_{I_{l}}}$ is about to occur; by time $\rho_l$ the event $B_{T_{I_{l}}}$ has also occurred. \lou{This in particular implies that at time $\rho_l=T_{I_l}+2$, the extent is less than $2\rad$, so since $L \ge 2\rad$ we have $T_{I_l+1}:=  \inf\{t\geq T_{I_l}+2:E(X(t))\leq L\} = T_{I_l}+2=\rho_l$.} Note that both $(\tau_i;i  \ge 1)$ and $(\rho_i;i\geq 1)$ are strictly increasing sequences, and that $(\rho_i;i\geq 1)$ are stopping times with respect to the filtration $(\cG_t,t \ge 0)$. 

\jl{For use in future proofs, we also introduce the stopping time 
\begin{equation*}
\tilde{T}^{(L)}:=\inf\left\{t\geq 0: E(X(t))\leq L\right\}, 
\end{equation*}
which can occur at time 0. The same argument as in the proof of Lemma \ref{lem:lem10} but with $\tilde T^{(L)}$ instead of $T^{(L)}$ shows that 
\begin{equation}\label{tildeTL}
\lim_{L\to\infty}\sup_{x\in \R^{d\times N}}\bP_x(\tilde T^{(L)}>L)=0.
\end{equation}
Similar to above, we set $\tilde{T}_1=\tilde{T}^{(L)}$, and recursively construct 
\begin{equation*}
\tilde{T}_{i+1}:=\inf\left\{t\geq \tilde{T}_{i}+2: E(X(t))\leq L\right\},\quad i\geq 1.
\end{equation*}
Letting 
\begin{equation*}
\tilde{I}_{1}:=\inf\left\{i\geq1: \text{$A_{\tilde{T}_{i}}\cap B_{\tilde{T}_{i}}$ occurs}\right\},
\end{equation*}
we define 
\begin{align}\label{e.tautildedef}
&\tilde{\tau}:=\tilde{T}_{\tilde{I}_{1}}+1\\
&\tilde{\rho}:=\tilde{T}_{\tilde{I}_{1}}+2\notag. 
\end{align}}

We now have the following key lemma, which guarantees that the times $\left(\tau_{i}; i\geq 1\right)$ satisfy the hypotheses of Proposition \ref{prop:prop2}. 
\begin{lem}\label{lem:main}
Assume $N\geq 3$. Let $\left(\tau_{i}; i\geq 1\right)$ be defined as above. Then for any $x \in \R^{d \times N}$, 
under $\bP_x$, $\rX$ and $\left(\tau_{i}; i\geq 1\right)$ satisfy the hypotheses of Proposition \ref{prop:prop2}.\end{lem}

To prove Lemma~\ref{lem:main}, we must verify properties (1)-(5) of Proposition \ref{prop:prop2}. Lemma~\ref{lem:main} follows immediately the results of the next two subsections. Specifically, Corollary \ref{c.properties1to3}, Proposition \ref{p:variance} and Proposition \ref{p.property5} establish properties (1)-(3), property (4), and property (5) from the hypotheses of Proposition \ref{prop:prop2}, respectively.
 
\begin{remark}\label{r:dependencies}
At first glance, it might seem as though a version of Lemma~\ref{lem:main} ought to hold for the times $(T_i;i \ge 1)$ in place of $(\tau_i;i \ge 1)$, but this is not the case. To see the sort of dependencies that can arise, note that if $T_{i+1} > T_i+2$ then the process necessarily passes through a state with large extent between time $T_i$ and $T_{i+1}$. In this case, with positive probability, at time $T_{i+1}$ the process has extent exactly $L$ (it could also be strictly less than $L$ if it reached ``extent at most $L$'' by a branching/killing event). 
On the other hand, if $T_{i+1}=T_i+2$ then with probability $1$ the extent at time $T_{i+1}$ strictly is less than $L$. This shows that the law of $X(T_{i+1})$ depends on the size of the increment $T_{i+1}-T_i$, and gives some indication of the subtlety of the dependencies at play. 
\end{remark}

\subsection{The regeneration structure}\label{s:renewal} 

The goal of this section is to show that, when considered on time intervals of the form $[\tau_m,\tau_{m+1}]$, the BBB process has an essentially IID structure. 
The key to revealing this structure is to only focus on descendants of particles of $X_1(\tau_m)$. Ignoring other particles creates some indexing issues, which have to be managed in order to avoid creating subtle dependencies. However, the gist of the story is that the events $A_{\tau_m-1}$ and $B_{\tau_m-1}$ together ensure that the descendants of particle $1$ take over the population on the time interval $[\tau_m,\tau_{m}+1]$; this is what yields the IID structure. \lou{Throughout the section, we work with the construction of the BBB as embedded within a BBM $\rW$, unless otherwise specified. }

Fix any $m \ge 1$, write $\sigma_m(1)=1$, and for $2 \le i \le N$ let $\sigma_m(i)$ be the index of the $(i-1)$'st BBB particle to be killed in the time interval $[\tau_m,\tau_m+1]$. Because the event $A_{\tau_{m}-1}\cap B_{\tau_m-1}$ occurs, all BBB particles aside from $X_1$ die during this time interval (Remark \ref{r.queen}), so $(\sigma_m(i))_{i \in [N]}$ is a permutation of $[N]$. 

Next, for $s \ge 0$, let $N_m(s)$ be the number of living time-$(\tau_m+s)$ descendants of $X_1(\tau_m)$, including $X_1(\tau_m+s)$ itself, in the BBB. The event \jl{$A_{\tau_{m}-1}\cap B_{\tau_m-1} = A_{T_{I_{m}}}\cap B_{T_{I_m}}$} ensures that $N_m(s) \ge 1$ for $s \ge 0$ and that $N_m(1)=N$, so the descendants of $X_1(\tau_m)$ have taken over the population by time $\tau_m+1=\rho_m$. It follows that that $N_m(s)=N$ for all $s \ge 1$. 

Now write $\hat{X}^{(m)}(s)=(\hat{X}^{(m)}_j(s))_{j\in [N_m(s)]}$ for the locations of time-$(\tau_m+s)$ descendants of $X_1(\tau_m)$ in the BBB, listed in order of their birth times; formally, $\hat{X}^{(m)}_j(s) = X_{\sigma_m(j)}(\tau_m+s)$. This means that for all $s \ge 0$, 
\[
\big(\hat{X}^{(m)}_j(s)-\hat{X}^{(m)}_1(0)\big)_{j\in [N_m(s)]} = 
\big(X_{\sigma_m(j)}(\tau_m+s)-X_1(\tau_m)\big)_{j \in [N_m(s)]}\, .
\]
Write $\rX^{(m)}=((\hat{X}^{(m)}_j(s)-\hat{X}^{(m)}_1(0))_{j\in [N_m(s)]},0 \le s \le \tau_{m+1}-\tau_m)$. The process $\rX^{(m)}$ describes the locations of BBB descendants of $X_1(\tau_m)$ relative to that of $X_1(\tau_m)$, between time $\tau_m$ and $\tau_{m+1}$. 


Finally, recall from Section~\ref{sub:formalism} that we allow for the BBB process to be started from a configuration with fewer than $N$ particles, and in this case we write $N(t)$ for the number of particles alive at time $t$. 
It is useful for the next proposition to introduce the notation $\tilde{\rX}:=(\tilde{X}(t),t \ge 0)$ with $\tilde{X}(t)=(\tilde{X}_i(t))_{i \in [N(t)]}$, for a BBB process started from a single particle at $0 \in \R^d$. We then write $F$ for the event that $\sup_{0 \le s \le 1} \max_{1 \le i \le N(t)} |\tilde{X}_i(s)| \le \rad$ and that $N(1)=N$.
\begin{prop}\label{p.IIDprocesses}
The processes $(\rX^{(m)};m \ge 1)$ are IID, and their common distribution is the conditional distribution of $(\tilde{X}(t),0 \le t \le \tau_1)$ given $F$. 
\end{prop}

\begin{proof}
{\color{cjl}
The main subtlety to the proof is that the events $A_{\tau_m-1} \cap B_{\tau_m-1}$ contain some information about the process on the time interval $[\tau_m,\rho_m]=[\tau_m,\tau_m+1]$, and the effect of this on the law must be understood. Recall that $(\cG_t,t \ge 0)$ is the natural filtration of the BBM $\rW=(W(t),t \ge 0)$. 

For any $m\geq 1$, by definition, $\tau_m\leq t$ if and only if $\rho_m\leq t+1$, which in turn occurs if and only if $T_{I_m}\leq t-1$ and $A_{T_{I_m}}\cap B_{T_{I_m}}$ occurs. 

The event $A_{\tau_m-1}$ is $\sigma(W(s),0 \le s \le \tau_m)$-measurable. 
The event $B_{\tau_m-1,1}\supset B_{\tau_m-1}$ enforces that $X_1(\tau_m)=W_{\mathcal{I}_1(\tau_m)}(\tau_m)$ and its descendants branch at least $N-1$ times in the time interval $[\tau_m,\tau_{m+1}]$.  Lemma~\ref{l.no_kill_center} and Remark~\ref{r.queen} together show that when $A_{\tau_m-1}$ and $B_{\tau_m-1}$ both occur, each such branching event increases the number of living descendants of $X_1(\tau_m)$ until those descendants occupy the full population. In particular, at time $\tau_{m}+1=T_{I_m}+2$, all BBB individuals are descendants of $X_1(\tau_m)$, and so $N_m(1)=N$. 

Moreover, the only constraint on the motion of the particles during the time interval $[\tau_m,\tau_{m}+1]$ is imposed by the event $B_{\tau_m-1,2} \supset B_{\tau_m-1}$; this event precisely requires that $|W_k(\tau_m+s)-W_{\cI_j(\tau_m)}(\tau_m)| \le \rad$ 
for all $0 \le s \le 1$, all $j \in [N]$ and all $k$ for which $W_k(\tau_m+s)$ is a descendant of $W_{\cI_j(\tau_m)}(\tau_m)=X_j(\tau_m)$. By the branching property of BBM, this conditioning affects the law of the trajectories of each BBB particle and its descendants independently. This application of the branching property is why we work with the  BBM rather than with the BBB directly.

The three preceding paragraphs imply in particular that conditionally given $(X(t),0 \le t \le \tau_m)$, the first unit of time of the restarted process, 
\[
(\hat{X}^{(m)}(s)-\hat{X}^{(m)}_1(0),0 \le s \le 1),
\] 
is distributed as the process $(\tilde{X}(t),0 \le t \le 1)$, conditioned on the event $F$ defined above the proposition statement. 

Next, $\rho_m=T_{I_m}+2$ is a $\cG_{t}$-stopping time, and  $A_{T_{I_m}}\cap B_{T_{I_m}} = A_{\rho_m-2}\cap B_{\rho_m-2}$ is $\cG_{\rho_m}$-measurable. Since 
\[
\hat{X}^{(m)}(1) = (X_{\sigma_m(j)}(\rho_m),j \in [N]), 
\]
it follows by the strong Markov property (together with the conclusion of the previous paragraph) that conditionally given $(X(t),0 \le t \le \tau_m)$, the process 
\[
(\hat{X}^{(m)}(s)-\hat{X}^{(m)}_1(0),s \ge 0)
\]
is distributed as $(\tilde{X}(t),t \ge 0)$ conditioned on the event $F$. In particular, since this law does not depend on $(X(t),0 \le t \le \tau_m)$, this implies that $(\hat{X}^{(m)}(s)-\hat{X}^{(m)}_1(0),s \ge 0)$ is independent of $(\rX^{(m')},1 \le m' < m)$. 


The last thing to check is that $\tau_{m+1}-\tau_m$ indeed has the correct law. Since $\tau_m=T_{I_m}+1$, necessarily $\tau_{m+1} \ge T_{I_m+1}\ge T_{I_m}+2 =\tau_m+1$. Our approach is  to first understand the law of $T_{I_m+1}$, so that we can think of the time interval $[\tau_{m}, \tau_{m+1}]$ as $[\tau_{m}, T_{I_{m}+1}]\cup [T_{I_{m}+1}, \tau_{m+1}]$. By definition, $T_{I_m+1} \ge T_{I_m}+2=\tau_m+1$; thus 
\[
T_{I_m+1}-\tau_m = \inf\{t \ge 1: E(\hat{X}^{(m)}(t)) \le L\}. 
\]
Since 
\[
T_1 = \inf\{t \ge 1: E(X(t)) \le L\}\, ,
\]
it follows that the law of $T_{I_m+1}-\tau_m$ is precisely the conditional law of $T_1$ in a BBB started from a single particle at $0 \in \R^d$ given that the event $F$ occurs.
(In fact, as noted after (\ref{eq:tau_def}), we have $T_{I_m+1}=T_{I_m}+2=\tau_m+1$, so $T_{I_m+1}-\tau_m=1$; but, indeed, when $F$ occurs, we also have $T_1=1$. However, we found it clearer to explain the equality in distribution without using this observation.) 

Finally, using that $T_{I_m+1}=T_{I_m}+2=\rho_m$ is a $\cG_t$-stopping time, the fact that $\tau_{m+1}-\tau_m$ and 
$(\hat{X}^{(m)}(s)-\hat{X}^{(m)}_1(0),s \ge 0)$ have the correct joint distribution now follows by the strong Markov property applied at time $T_{I_m+1}$ together with the definition of $\tau_{m+1}$. }
\end{proof} 

\begin{cor}\label{c.properties1to3}
For any $x \in \R^{d\times N}$, under $\bP_x$, the BBB process and the random times $(\tau_i;i \ge 1)$ satisfy conditions (1), (2) and (3) of Proposition~\ref{prop:prop2}. 
\end{cor}
\begin{proof}
The differences $(\tau_{m+1}-\tau_m,m \ge 1)$ are the total lifetimes of the processes $(\rX^{(m)};m \ge 1)$; since these processes are IID, so are the differences $(\tau_{m+1}-\tau_m,m \ge 1)$. \jl{Moreover, since 
\begin{equation*}
\tau_{2}-\tau_{1}=(T_{I_{2}}+1)-(T_{I_{1}}+1)\geq (T_{I_{1}}+3)-(T_{I_{1}}+1)\geq 2,
\end{equation*}
 we see that $\bE_{x}[\tau_{2}-\tau_{1}]$ is nonzero.} This verifies property (1).

Next, writing 
$\jl{\rX_1^{(m)}}=(X_1^{(m)}(s),0 \le s \le \tau_{m+1}-\tau_m)$ for the projection of $\rX^{(m)}$ on its first coordinate, then the processes $(\rX_1^{(m)},m \ge 1)$ are also IID. By definition, 
\[
\rX_1^{(m)}(s) = X_{\sigma_m(1)}(\tau_m+s)-X_1(\tau_m) = X_1(\tau_m+s)-X_1(\tau_m)\, ,
\]
the second inequality since $\sigma_m(1)=1$ for all $m \ge 1$. It follows that the processes 
\[
(X_1(\tau_m+s)-X_1(\tau_m),0 \le s \le \tau_{m+1}-\tau_m)
\]
are IID for $m \ge 1$; properties (2) and (3) of the proposition are immediate. 
\end{proof}
\jl{\begin{remark}
One also might try to recover an IID structure, and thereby simplify the proof, by using the stopping times $(\rho_i; i \ge 1)$ rather than $(\tau_i;i \ge 1)$. However, at the times $\rho_i$, there are other dependencies. To define $\rX^{(i)}$ in Section~\ref{s:renewal}, when particle $X_1$ starts to reproduce at time $\tau_i$, we label the offspring in increasing order of birth time until there are $N$ of them. This is not how they would be labeled in the underlying BBB process, which is why we need the random permutation $\sigma^{(i)}$ to describe the relabeling. To get an IID structure for $\rX$ itself at the regeneration times $\rho_i$, we would need to know that the permutations $(\sigma^{(i)};i \ge 1)$ are IID. We have not proved such independence and, indeed, it does not seem obvious to us that it should hold. 
\end{remark}}

\subsection{Finite variance of the increments}\label{s:variance} In this section, we prove property $(4)$ in Proposition \ref{prop:prop2}. We first establish that the regeneration increments have exponential tails. 

\begin{lem}\label{lem:lem11} There exists $b=b_{L,N} >1$ such that 
\[
\sup_{x \in \R^{d \times N}} \bE_x[b^{\tilde{\rho}}]<\infty
\quad\text{and}
\quad
\sup_{x \in \R^{d \times N}} \bE_x[b^{\rho_{1}}]<\infty.
\]
\end{lem} 
\begin{proof} 
We first prove the first bound, then explain the small modifications to the proof that are required to prove the second bound. Fix any initial configuration $x\in \R^{d\times N}$. For any $b > 1$, using the change of variables $r\coloneqq \log_b(s)$, we may write
\begin{align*}
\bE_x[b^{\tilde \rho}]&=\int_0^\infty \bP_x(\tilde \rho>\log_b(s))ds\\&=\ln(b)\left(\int_{-\infty}^1 b^r\bP_x(\tilde \rho>r)dr+\int_1^\infty b^r\bP_x(\tilde \rho>r)dr\right).
\end{align*}
The first integral is bounded from above by $\ln(b)\int_{-\infty}^1 b^r dr$ which is finite for any $b>1$. By \eqref{tildeTL}, there exists $L>1$ such that 
\begin{equation}\label{e.taulbnd}
\inf_{x\in \R^{d\times N}} \bP_x(\tilde T_1\leq L)>1/2.
\end{equation}
Moreover, we claim Lemma~\ref{lem:lem6} implies the existence of a constant $c=c_{L,N}>0$ such that for all $t\geq L+2$,
\begin{equation}\label{e.rholbnd}
\inf_{x\in \R^{d\times N}}\bP_x(\tilde \rho\leq t|\tilde T_1\leq L)\ge c.
\end{equation}
\lou{To see this, note that if $A_{\tilde T_1} \cap B_{\tilde T_1}$ occurs then $\tilde \rho=\tilde T_1+2$, so if additionally $\tilde T_1 \le L$ then $\tilde \rho \le L+2 \le t$. 
 It follows that for all $t\geq L+2$, for any $x \in \R^{d \times N}$, 
\begin{align*}
\bP_x(\tilde \rho\leq t|\tilde T_1\leq L)
& \ge \inf_{y\in C_L}\bP_y(A_{0}\cap B_{0})
\end{align*}
where in the last step we have used the Markov property at time $T_1$, together with the fact that $X(\tilde T_1) \in C_L$. }

Denote $c'\coloneqq c/2\in (0,1)$. By (\ref{e.taulbnd}) and (\ref{e.rholbnd}), for any $t\geq L+2$ and any $x \in \R^{d \times N}$, 
\begin{align*}
\bP_x(\tilde\rho\leq t)&\geq \bP_x(\tilde\rho\leq t, \tilde T_1\leq L)\\&\geq \bP_x(\tilde\rho\leq t|\tilde T_1\leq L)\bP_x(\tilde T_1\leq L)\\&> c',
\end{align*} 
\jl{  For all $t\geq 0$, it follows from the Markov property at time $t$ that
\begin{align}
\bP_x(\tilde\rho> t+L+2|\tilde \rho>t) &\notag= \bP_x(\tilde \rho> t)^{-1}\bE_x\left[ \bP_x(\tilde \rho>t+L+2|\cG_t)\mathbbm{1}_{\{\tilde \rho>t\}}\right]\\&\notag=\bE_x\left[\bP_{X(t)}(\tilde \rho>L+2)|\tilde \rho>t\right]\\&\leq \sup_{y\in \R^{d\times N}}\bP_y(\tilde \rho>L+2)<1-c',\label{diffForRho1}
\end{align}
for all $x\in \R^{d\times N}$. We obtain by induction that
\[
\bP_x(\tilde \rho>t)\leq \prod_{k=1}^{\lfloor t/(L+2)\rfloor}\bP_x(\tilde \rho>k(L+2)|\tilde \rho>(k-1)(L+2))\leq (1-c')^{\lfloor t/(L+2)\rfloor}
\]
for all $x\in \R^{d\times N}$.
Thus,
\[
\int_1^\infty b^r\bP_x(\tilde \rho>r)dr \leq \int_1^\infty b^r (1-c')^{\lfloor r/(L+2) \rfloor} dr.
\]}
The integral is finite provided we choose $b>1$ such that $b(1-c')^{1/(L+2)}<1$. Since $c'\in (0,1)$, this is possible. 

\jl{ 
To establish the second bound of the lemma, essentially the same argument applies, using $T_1$ instead of $\tilde T_1$ and Lemma \ref{lem:lem10} instead of \eqref{tildeTL}. The only calculation which changes is the analogue of the bound in \eqref{diffForRho1}. The equivalent estimate for $\rho_1$ follows by applying the Markov property at time $t \ge 1$ to obtain that 
\begin{align*}
\bP_x(\rho_1>t+L+2|\rho_1>t)&\notag=\bE_x[\bP_x(\rho_1>t+L+2|\cG_t)|\rho_1>t]\\&\notag=\bE_x[\bP_{X(t)}(\tilde \rho>L+2)|\rho_1>t]\\&\leq \sup_{y\in \R^{d\times N}}\bP_y(\tilde \rho>L+2)<1-c',
\end{align*}
for all $x\in \R^{d\times N}$.}
\end{proof}
Note that, taking $\rho_0\equiv0$, the lemma implies that 
\[
\sup_{x\in \R^{d\times N}}\bE_x[b^{\rho_{i+1}-\rho_i}]<\infty
\] for any $i\geq 0$, since we may write by the tower law and the strong Markov property \jl{ 
\[
\bE_x[b^{\rho_{i+1}-\rho_i}]=\bE_x[\bE_x[b^{\rho_{i+1}-\rho_i}|\cG_{\rho_i}]]=\bE_x[\bE_{X(\rho_i)}[b^{\tilde \rho}]]
\] for all $i\geq 1$,} and Lemma~\ref{lem:lem11} implies that the inner conditional expectation is deterministically bounded by some $C<\infty$. Also observe that $\tau_{i+1}-\tau_i=\rho_{i+1}-\rho_i$ for each $i\geq 1$ and $\tau_1\leq \rho_1$, so Lemma~\ref{lem:lem11} yields
\begin{equation}\label{e.expTails_tau}
\sup_{x\in \R^{d\times N}}\bE_x[b^{\tau_{i+1}-\tau_i}]<\infty.
\end{equation} for all $i\geq 0$, where we have defined $\tau_0\equiv 0$. This last fact will be used in the proof of property (4) of Proposition \ref{prop:prop2} for the BBB process (Proposition \ref{p:variance}).

\begin{prop}\label{p:variance}
For any $x \in \R^{d\times N}$, under $\bP_x$, the BBB process and the stopping times $(\tau_i;i \ge 1)$ satisfy condition (4) of Proposition~\ref{prop:prop2}. 
\end{prop}
\begin{proof}
The condition asserts that 
   \[
    \bE_x\left[\sup_{\tau_1\leq t\leq \tau_2}|X_1(t)-X_1(\tau_1)|^2\right]<\infty.
   \]
We bound the supremum as follows. Recall that $\rho_1=\tau_1+1$, and write 
\begin{align*}
\sup_{\tau_1\leq t\leq \tau_2}|X_1(t)-X_1(\tau_1)|^2
& =
\max\Big(
\sup_{\tau_1\leq t\leq \rho_1}|X_1(t)-X_1(\tau_1)|^2,
\sup_{\rho_1\leq t\leq \tau_2}|X_1(t)-X_1(\tau_1)|^2
\Big)\\
& \le 
\sup_{\tau_1\leq t\leq \rho_1}|X_1(t)-X_1(\tau_1)|^2+ \sup_{\rho_1\leq t\leq \tau_2}|X_1(t)-X_1(\tau_1)|^2. \\
& \le 
\rad + \sup_{\rho_1\leq t\leq \tau_2}|X_1(t)-X_1(\tau_1)|^2\, ,
\end{align*}
the last inequality since by definition, $|X_1(t)-X_1(\tau_1)| \le \rad$ for $t \le \tau_1+1=\rho_1$. 
So to prove the lemma it suffices to show that 
\begin{equation}\label{e.2mombd}
\bE_x\left[\sup_{\rho_1\leq t\leq \tau_2}|X_1(t)-X_1(\tau_1)|^2\right] < \infty\, .
\end{equation}

Now write $\pi$ for the law of \jl{$X(\rho_{1})-X_1(\tau_1)$}. Precisely, if $\rX$ is the BBB process started from a single particle at $0 \in \R^d$, then $\pi$ is the conditional law of $X(1)$ given that $N(1)=N$ and that 
\[
\sup_{0 \le t \le 1} \max_{1 \le i \le N(t)} \left|X_i(t)\right| \le \rad\, .
\] 
Recall that $\tau_1-1 = T_{I_1}$ is the time at which $A_{T_{I_1}} \cap B_{T_{I_1}}$ occurs. By definition, $A_{T_{I_1}} \cap B_{T_{I_1}}$ is $\cF_{\rho_1}$-measurable, since $\rho_1=\tau_1+1=T_{I_1}+2$. \jl{ 
With this definition, applying the strong Markov property at time $\rho_{1}$ yields 
\[
 \bE_x\left[\sup_{\rho_1\leq t\leq \tau_2}|X_1(t)-X_1(\tau_1)|^2\right] = 
\bE_\pi\left[\sup_{0 \le t \le \tilde{\tau}} |X_1(t)|^2\right],   
\]
for $\tilde{\tau}$ defined according to \eqref{e.tautildedef}. In particular, we have 
\begin{align*}
 \bE_x\left[\sup_{\rho_1\leq t\leq \tau_2}|X_1(t)-X_1(\tau_1)|^2\right] &=\bE_\pi\left[\sup_{0 \le t \le \tilde{\tau}} |X_1(t)-X_{1}(0)+X_{1}(0)|^2\right]\\
&\leq 2\bE_\pi\left[\sup_{0 \le t \le \tilde{\tau}} |X_1(t)-X_{1}(0)|^{2}\right]+2\bE_\pi\left[|X_{1}(0)|^{2}\right]\\
&\leq 2\bE_\pi\left[\sup_{0 \le t \le \tilde{\tau}} |X_1(t)-X_{1}(0)|^{2}\right]+2\rad^{2} 
\end{align*}
where we have used the inequality $(a+b)^2 \le 2a^{2}+2b^{2}$ and the fact that since $|X_1(\rho_1)-X_1(\tau_1)| \le\rad$, whence $X_{1}(0)$ started from an initial condition distributed as $\pi$ must also satisfy $|X_{1}(0)|\leq \rad$.}

In view of the preceding bound, in order to verify \eqref{e.2mombd}, it suffices to show that
\begin{equation*}
\bE_\pi\left[\sup_{0 \le t \le \tilde{\tau}} |X_1(t)-X_{1}(0)|^{2}\right]<\infty. 
\end{equation*}
To do so, we make use of the first construction of the BBB process given in Section \ref{sub:formalism}. In other words, we view $\rX$ as constructed from $N$ independent Brownian motions $(B_i)_{i \in [N]}$ with the aid of a Poisson counting process $\cJ$ which gives the branching times, and of a sequence $(\cU_i;i\geq 1)$ of random variables which determines which particle branches at each branching time.
Since the ancestral trajectories are continuous paths, an immediate consequence of this construction is that 
{\color{cjl}
\begin{equation}\label{e.BBBbd}
\sup_{0\leq t\leq \jl{\tilde{\tau}}} \max_{j\in[N]}|X_j(t)-X_j(0)|\leq 2\cJ_{\jl{\tilde{\tau}}}\sup_{0\leq s\leq \jl{\tilde{\tau}}}\max_{j\in [N]}|B_j(s)|+\max_{i, j\in [N]}|X_i(0)-X_j(0)|\, ,
\end{equation}}
so it suffices to prove that 
\begin{equation}\label{eq:jb_bound}
\bE_\pi\left[\cJ_{\jl{\tilde{\tau}}}^2\sup_{0\leq s\leq \jl{\tilde{\tau}}}\max_{j\in [N]}|B_j(s)|^2\right]<\infty; 
\end{equation}
establishing \eqref{eq:jb_bound} occupies the remainder of the proof. 

By the Cauchy-Schwarz inequality, we have 
\begin{equation}\label{e.toBdVar}
\bE_\pi\left[\cJ_{\jl{\tilde{\tau}}}^2\sup_{0\leq s\leq \jl{\tilde{\tau}}}\max_{j\in [N]}|B_j(s)|^2\right]\leq \bE_\pi[\cJ_{\jl{\tilde{\tau}}}^4]^{1/2}\bE_\pi\left[\sup_{0\leq s\leq \jl{\tilde{\tau}}}\max_{j\in [N]}|B_j(s)|^4\right]^{1/2}
\end{equation}
Moreover, since $t\mapsto \cJ_t$ is almost-surely nondecreasing
\begin{align}\label{upBdPPP}
\bE_\pi\left[\cJ_{\jl{\tilde{\tau}}}^4\right]&=\bE_\pi\left[\sum_{p=0}^\infty\cJ_{\jl{\tilde{\tau}}}^4\mathbbm{1}_{\{\jl{\tilde{\tau}}\in [2^p, 2^{p+1}]\}}\right]\notag\\&\leq \sum_{p=0}^\infty \bE_\pi\left[\cJ_{2^{p+1}}^4\mathbbm{1}_{\{\jl{\tilde{\tau}}\in [2^p, 2^{p+1}]\}}\right]\notag\\&\leq \sum_{p=0}^\infty \bE_\pi\left[\cJ_{2^{p+1}}^8\right]^{1/2}\bP_\pi\left(\jl{\tilde{\tau}}\geq 2^p\right)^{1/2}.
\end{align}
Since $\tilde{\tau} \le \tilde\rho$, by Markov's inequality and Lemma~\ref{lem:lem11} there exist constants $b>1$ and $K>0$ such that for all $p \ge 0$, 
\begin{equation}\label{e.expTailBd}
\bP_\pi\left(\jl{\tilde{\tau}}\geq 2^p\right)\leq Kb^{-2^p}.
\end{equation}
A standard computation of higher moments of Poisson random variables yields
\[
\bE_\pi\left[\cJ^8_{2^{p+1}}\right]=\sum_{i=0}^8N^i2^{i(p+1)}\stirling{8}{i}=O(N^82^{8(p+1)}),
\]
where $\stirling{j}{i}$ denotes a Stirling number of the second kind.
Continuing from \eqref{upBdPPP}, we obtain
\begin{equation}\label{e.1stterm}
\bE_\pi\left[\cJ_{\jl{\tilde{\tau}}}^4\right]\leq K^{1/2}N^4\sum_{p=0}^\infty O(2^{4(p+1)})b^{-2^{p-1}}<\infty.
\end{equation}
 To bound the second factor in (\ref{e.toBdVar}), we note that
\begin{align}\label{e.2ndterm1}
\bE_\pi\left[\sup_{0\leq s\leq \jl{\tilde{\tau}}}\max_{i\in [N]}|B_i(s)|^4\right]&=\sum_{p=0}^\infty \bE_\pi\left[\sup_{0\leq s\leq \jl{\tilde{\tau}}}\max_{i\in [N]}|B_i(s)|^4\mathbbm{1}_{\{\jl{\tilde{\tau}}\in [2^p, 2^{p+1}]\}}\right]\notag\\&\leq \sum_{p=0}^\infty \bE_\pi\left[\sup_{0\leq s\leq 2^{p+1}}\max_{i\in [N]}|B_i(s)|^4\mathbbm{1}_{\{\jl{\tilde{\tau}}\in [2^p, 2^{p+1}]\}}\right]\notag\\& \leq \sum_{p=0}^\infty \bP_\pi\left(\jl{\tilde{\tau}}\geq 2^p\right)^{1/2}\bE_\pi\left[\sup_{0\leq s\leq 2^{p+1}}\max_{i\in [N]}|B_i(s)|^8\right]^{1/2}.
\end{align} 
 Observe that 
 \[
 \max_{i\in [N]}|B_i(s)|^8\leq d^4\max_{i\in [N]}\max_{j\in [d]}|B_{ij}(s)|^8,
 \]
 and that
\[
\left(\max_{i\in [N]}\max_{j\in [d]}|B_{ij}(t)|, t\geq 0\right)
\] is a submartingale. So Doob's $L^8$ inequality and IID-ness of the one-dimensional Brownian motions $\rB_{ij}$ imply that 
\begin{align}\label{e.2ndterm}
\bE_\pi\left[\sup_{0\leq s\leq 2^{p+1}}\max_{i\in [N]}|B_i(s)|^8\right]&\leq d^4\left(\frac{8}{7}\right)^8 \bE_\pi\left[\max_{i\in [N]}\max_{j\in [d]}|B_{ij}(2^{p+1})|^8\right]\notag\\&\leq d^5N\left(\frac{8}{7}\right)^8 \bE_\pi\left[|B_{11}(2^{p+1})|^8\right]\notag\\&=d^5NO(2^{4(p+1)}),
\end{align} 
where we have used that $\bE_\pi\left[|B_{11}(2^{p+1})|^8\right]=O\left(2^{4(p+1)}\right)$. Combining \eqref{e.2ndterm1}, \eqref{e.2ndterm} and \eqref{e.expTailBd}, we conclude that 
\begin{equation*}\label{e.2ndtermreal}
\bE_\pi\left[\sup_{0\leq s\leq \jl{\tilde{\tau}}}\max_{i\in [N]}|B_i(s)|^4\right]\leq K^{1/2}d^5N\sum_{p=0}^\infty b^{-2^{p-1}}O(2^{4(p+1)})<\infty\, ,
\end{equation*}
which establishes \eqref{eq:jb_bound}.
\end{proof}

\subsection{Barycenter approximation in probability}
In this section, we show that property (5) of Proposition \ref{prop:prop2} holds for the BBB process. We will make frequent use of the following fact on the convergence in probability of certain random variables. 
\begin{fact}\label{l.convProb}
Let $(k_m)_{m\geq 1}$ be a sequence of nonnegative integer valued random variables with the property that there exists a constant $\kappa>0$ such that $m^{-1}k_m\to \kappa$ in probability as $m \to \infty$. Suppose that $(Z_i;i\geq 1)$ are identically distributed nonnegative random variables with $\E{Z_1^2}<\infty$. Then, as $m\to\infty$
\begin{equation}\label{e.convInProb}
m^{-1/2}\max_{1\leq i< k_m} Z_i\to 0
\end{equation} in probability.
\end{fact}
\begin{proof}
Since $Z_1\in L^2(\bP)$, it is $\bP$-almost surely finite. Hence, by Chebyshev's inequality and the dominated convergence theorem, we obtain
\begin{equation}\label{e.Z0upbd}
\lim_{m\to\infty}m\bP(m^{-1/2}Z_1>\epsilon)\leq \lim_{m\to\infty} \bE\left[\epsilon^{-2}Z_1^2\mathbbm{1}_{m^{-1/2}Z_1>\epsilon}\right]=0
\end{equation}
We have 
 \begin{align}\label{e.lim_toCompute}
&\bP\left(m^{-1/2}\max_{1\leq i< k_m}Z_i>\epsilon \right)\notag\\&\leq \bP\left(m^{-1/2}\max_{1\leq i< \epsilon^{-1}m}Z_i>\epsilon \Big|k_m\leq \epsilon^{-1}m\right)+\bP\left(k_m>\epsilon^{-1}m\right).
\end{align} 
For all $\epsilon>0$ small enough that $\epsilon^{-1}>\kappa$, since $m^{-1}k_m\to\kappa $ in probability as $m\to\infty$ by assumption, it follows that 
\begin{equation}\label{e.lim1}
\lim_{m\to\infty}\bP\left(k_m>\epsilon^{-1}m\right)=0.
\end{equation}
Moreover, since the $Z_i$ are identically distributed, a union bound and \eqref{e.Z0upbd} give
\begin{equation}\label{e.lim2}
\bP\left(m^{-1/2}\max_{1\leq i<\epsilon^{-1}m}Z_i>\epsilon \Big|k_m\leq \epsilon^{-1}m\right)\leq \epsilon^{-1}m \bP(m^{-1/2}Z_1>\epsilon)\xrightarrow{m\to\infty} 0.
\end{equation}
Combining \eqref{e.lim_toCompute}, \eqref{e.lim1} and \eqref{e.lim2} establishes \eqref{e.convInProb}.
\end{proof} 
For each $m\geq \tau_1$, let $k[m]$ denote the unique natural number such that $\tau_{k[m]}\leq m< \tau_{k[m]+1}$, and for $m < \tau_1$ let $k[m]=0$. By Corollary \ref{c.properties1to3}, the increments $(\tau_{k+1}-\tau_{k};k \ge 1)$ are IID; since $\tau_{k+1}-\tau_k = \rho_{k+1}-\rho_k$, by Lemma~\ref{lem:lem11} these increments have some positive finite expected value $\kappa \in (0,\infty)$.  The law of large numbers then implies that $\tau_k/k \convas \kappa$ as $k \to \infty$, from which it follows easily that $m^{-1}k[m] \convas 1/\kappa$ as $m \to \infty$. 
\begin{prop}\label{p.property5}
Fix any $x\in \R^{d\times N}$. As $m\to\infty$, we have
\begin{equation}\label{e.goalConv}
m^{-1/2}\sup_{0\leq t\leq m}|\overline{X}(t)-X_1(t)|\to 0
\end{equation} in probability with respect to $\bP_x$. 
\end{prop}
\begin{proof}
Since $\tau_1<\infty$ almost surely with respect to $\bP_x$ by Lemma \ref{lem:lem11}, it suffices to prove that as $m\to\infty$
\[
m^{-1/2}\sup_{\tau_1\leq t\leq m}|\overline{X}(t)-X_1(t)|\to 0
\] in probability with respect to $\bP_x$.
For all $t\geq 0$, we have 
\[
|\overline{X}(t)-X_1(t)|\leq \max_{j\in [N]}|X_j(t)-X_1(t)|.
\] Therefore 
\[
\sup_{\tau_1\leq t\leq m} |\overline{X}(t)-X_1(t)|\leq \max_{i\in [k[m]]}\sup_{\tau_i\leq s\leq \tau_{i+1}}\max_{j\in [N]}|X_j(s)-X_1(s)|.
\] 
At time $\tau_i=T_{I_i}+1$, the event $A_{T_{I_i}}$ has just occurred, which implies that $\max_{j \in [N]} |X_j(\tau_i)-X_1(\tau_i)| \le 5+2^{-1}(N+1)^{-1}+5(N-1)^{-1}< 11$. Moreover, the event $B_{T_{I_i}}$ is about to occur, which means that for all times $s \in [\tau_i,\rho_i]=[\tau_i,\tau_{i}+1]$, every particle has distance at most $\rad=\frac{1}{4(N+1)}$ from the location of some time-$\tau_i$ particle.  Therefore, 
\[
\sup_{\tau_i\leq s\leq \rho_i}\max_{j\in [N]}|X_j(s)-X_1(s)| \le 12\, .
\]
Writing 
\[
Z_i\coloneqq \max\left(12,\sup_{\rho_i\leq s\leq \tau_{i+1}}\max_{j\in [N]}|X_j(s)-X_1(s)|\right)
\]
for $i \ge 1$, it follows that 
\[
\sup_{\tau_1\leq t\leq m} |\overline{X}(t)-X_1(t)|\leq \max_{i\in [k[m]]} Z_i. 
\]
To obtain (\ref{e.goalConv}), it thus suffices to verify the conditions of Fact \ref{l.convProb} for $(Z_i;i\geq 1)$. 
Under $\bP_x$, the random variables $\sup_{\rho_i\leq s\leq \tau_{i+1}}\max_{j\in [N]}|X_j(s)-X_1(s)|$ are identically distributed for $i\in [k[m]]$ by Proposition \ref{p.IIDprocesses}, from which it is immediate that $(Z_i;i \geq 1)$ are identically distributed. 
 It only remains to show that $Z_1\in L^2(\bP_x)$. For this, observe that 
 {\color{cjl} \begin{align*}
 \max_{j\in [N]}|X_j(s)-X_1(s)|^2&\leq 2\max_{j\in [N]}|X_j(s)-X_1(\tau_1)|^2+2|X_1(s)-X_1(\tau_1)|^2\\&\leq 4\max_{j\in [N]}|X_j(s)-X_1(\tau_1)|^2,
 \end{align*}

 for all  $s\in[\rho_1,\tau_2]$. Taking the supremum over $s\in[\rho_1,\tau_2]$ and taking expectations, we infer that 
  \[
 \bE_x[Z_1^2]\leq 144+4\bE_x\left[\sup_{\rho_1\leq s\leq \tau_2}\max_{j\in [N]}|X_j(s)-X_1(\tau_1)|^2\right]. 
 \]
 By the strong Markov property and (\ref{e.BBBbd}), we obtain
 \begin{align*} 
 \bE_x&\left[\sup_{\rho_1\leq s\leq \tau_2}\max_{j\in [N]}|X_j(s)-X_1(\tau_1)|^2\right]\\
 &\quad=\bE_\pi\left[\sup_{0\leq s\leq \tilde \tau}\max_{j\in [N]}|X_j(s)-X_j(0)+X_j(0)|^2\right]\\&\quad\leq 2\bE_\pi\left[\sup_{0\leq s\leq \tilde \tau}\max_{j\in [N]}|X_j(s)-X_j(0)|^2\right]+2\bE_\pi\left[\max_{j\in [N]}|X_j(0)|^2\right]\\&\quad\leq 4\bE_\pi\left[4\cJ_{\tilde \tau}^2\sup_{0\leq s\leq \tilde \tau}\max_{j\in [N]}|B_j(s)|^2\right]+4\bE_\pi\left[\max_{i, j\in [N]}|X_j(0)-X_i(0)|^2\right]\\&\quad+2\bE_\pi\left[\max_{j\in [N]}|X_j(0)|^2\right]\\&\quad\leq 16 \bE_\pi\left[\cJ_{ \tilde \tau}^2\sup_{0\leq s\leq \tilde \tau}\max_{j\in[N]}|B_j(s)|^2\right]+ 10\rad^2,
 \end{align*}} which is finite by \eqref{eq:jb_bound}. This proves that $Z_1\in L^2(\bP_x)$; and establishes the convergence in (\ref{e.goalConv}). 
\end{proof}
\subsection{Proof of Theorem \ref{thm:main}}\label{subsec:Drift}
Set $\Delta_1\coloneqq X_1(\tau_2)-X_1(\tau_1)$, and denote by $\nu$ the law of $\Delta_1$ under $\bP_x$; it is a Borel probability measure on $\R^d$ defined by 
\color{cjl}
\begin{align*}
\nu(B) & \coloneqq \bP_x(X_1(\tau_2)-X_1(\tau_1)\in B)\\
& = \bP_{X(0)=0}\Big(X_1(\tau_1)\in B~\Big|\sup_{0 \le s \le 1} \max_{1 \le i \le N(s)} |X_i(s)| \le \rad,N(1)=N\Big)\, ,
\end{align*}
where $\bP_{X(0)=0}$ denotes the probability measure under which $\rX=(X(t),t \ge 0)$ is a BBB started from a single particle at the origin. 
\color{black}
The second equality holds by Proposition \ref{p.IIDprocesses} and makes it clear that the measure $\nu$ does not depend on the initial configuration $x\in \R^{d\times N}$. 
The rotational invariance of Brownian motion immediately implies the following lemma (whose proof we omit).
\begin{lem}\label{lem:lem12} 
The measure $\nu$ is invariant under the action of the group $O(d, \R)$ of $d$-by-$d$ orthogonal matrices with entries in $\R$. 
\end{lem}
We are now ready to prove Theorem \ref{thm:main}.
\begin{proof}[Proof of Theorem \ref{thm:main}]
Fix $x\in \R^{d\times N}$. We work on the probability space $(\Omega, \cF, \bP_x)$. 
When $N=1$, the process $\rX$ consists of a single Brownian motion so the theorem holds with $\sigma(d, 1)=1$ and no drift term. 

Next suppose that $N=2$. In this case, the first construction of the BBB process given in Section~\ref{sub:formalism} simplifies as follows. Let $(\zeta_i;i\geq 1)$ be the jump times for a counting process $\cJ= (\cJ_t, t\geq 0)$ for a Poisson point process on $[0,\infty)$ with rate $2$ and define $\zeta_0=0$. We denote by $\rX_1=(X_1(t), t\geq 0)$ and $\rB=(B(t), t\geq 0)$ two $d$-dimensional Brownian motions started from $x_1\in \R^d$ and $0\in \R^d$, respectively. Define a process $\rX_2=(X_2(t), t\geq 0)$ by $X_2(0)=x_2\in \R^d$, and inductively for each $i\geq 0$ by
\[
X_2(t)=X_1(\zeta_i)+B(t)-B(\zeta_i), \;\;\;t\in [\zeta_i, \zeta_{i+1}).
\]
Thus, when $N=2$, the BBB process $\rX=(\rX_1, \rX_2)$ consists of a Brownian motion $\rX_1$ to which we attach branches $\rX_2$ that can only grow between consecutive branching times. For $s\geq 0$, let $k[s]$ be the unique positive integer such that $\zeta_{k[s]}\leq s\leq \zeta_{k[s]+1}$. By the triangle inequality, we have
\begin{align*}
&\sup_{0\leq t\leq 1}\left|\frac{\overline{X}(t m)}{\sqrt{m}} - \frac{X_1(t m)}{\sqrt{m}}\right|\\&=\sup_{0\leq t\leq 1}\left|\frac{X_2(t m)}{2\sqrt{m}} - \frac{X_1(t m)}{2\sqrt{m}}\right|\\&\leq \sup_{0\leq t\leq 1}\left|\frac{X_2(t m)}{2\sqrt{m}} - \frac{X_2(\zeta_{k[t m]})}{2\sqrt{m}}\right|+ \sup_{0\leq t\leq 1}\left|\frac{X_1(t m)}{2\sqrt{m}} - \frac{X_1(\zeta_{k[t m]})}{2\sqrt{m}}\right|\\&\leq \max_{0\leq i<k[m]}\sup_{s\in [\zeta_i, \zeta_{i+1}]}\frac{|X_2(s)-X_2(\zeta_i)|}{2\sqrt{m}}+\max_{0\leq i<k[m]}\sup_{s\in [\zeta_i, \zeta_{i+1}]}\frac{|X_1(s)-X_1(\zeta_i)|}{2\sqrt{m}}.
\end{align*}
We claim that the last expression converges to $0$ in probability. Since $\zeta_2-\zeta_1$ has exponential tails, and for each $j\in \{1,2\}$,
\[
(X_j(s)-X_j(\zeta_1), s\in [\zeta_1, \zeta_2])
\] is a Brownian motion, we infer that
\[
\sup_{s\in [\zeta_1,\zeta_2]}|X_j(s)-X_j(\zeta_1)|\in L^2(\bP_x).
\] 
Therefore, Fact \ref{l.convProb} implies the convergence to $0$ in probability. Since the initial value of $\overline{\rX}$ is sent to $0$ when $m\to\infty$, we conclude that $\overline{\rX}$ satisfies an invariance principle with $\sigma(d, 2)=1$.

Suppose now that $N\geq 3$. By Lemma~\ref{lem:main} and Proposition \ref{prop:prop2} the barycenter $\overline{\rX}$ satisfies an invariance principle, i.e. there exists a non-zero $d$-by-$d$ matrix $\Sigma=\Sigma(d,N)$ such that \eqref{e.geninv} holds. We next show that the drift term in the invariance principle is zero for any $d\geq 1$ and $N\geq 3$. By Proposition \ref{prop:prop2}, the drift is $\alpha= \bE_x[\Delta_1]\bE_x[\tau_2-\tau_1]^{-1}$. Moreover, we note that property (4) in Proposition \ref{prop:prop2}, combined with the Cauchy-Schwarz inequality, immediately implies that $\Delta_i=X_1(\tau_{i+1})-X_1(\tau_i)\in L^1(\bP_\pi)$, for each $i\geq 1$. In particular, $\bE_x[\Delta_1]<\infty$. The rotational invariance from Lemma~\ref{lem:lem12} then immediately implies that $\bE_x[\Delta_1]=0$, and hence $\alpha=0$. Finally, since for any $m>0$, the process $\left(m^{-1/2}\overline{X}(tm),0\leq t\leq 1\right)$ is rotationally invariant, then the limit process $(\Sigma B(t), 0\leq t\leq 1)$ is also rotationally invariant. 
Considering rotations of the form $(x,y)\mapsto (-y, x)$ applied to all pairs of coordinates, rotational invariance implies that $(\Sigma B(t), 0\leq t\leq 1)$ has uncorrelated and identically distributed coordinates. Since $(\Sigma B(t), 0\leq t\leq 1)$ is a Gaussian process, this implies its coordinates are in fact IID, and thus $\Sigma(d,N)=\sigma(d,N)I_d$ for some $\sigma(d,N)>0$. 
\end{proof}

\section{Proofs of Technical Ingredients}\label{sec:properties}

In this section, we include the proofs of Lemma~\ref{lem:lem6}, Lemma~\ref{lem:lem10}, and Theorem \ref{thm:HarrisRec}. Recall from Section~\ref{subsec:ConstrOutline} that 
$C_L=\{x\in \R^{d\times N}: E(x)\leq L\}$ is the set of configurations with extent at most $L$.

\subsection{Two events with positive probability}\label{subsec:2events} We begin by proving that the events $A$ and $B$ defined by \eqref{e.at1}-\eqref{e.at3} and \eqref{e.bt1},\eqref{e.bt2} have positive probability.
\begin{lem} \label{lem:lem3}
Let $A=A_0=A_{0,1}\cap A_{0,2}\cap A_{0,3}$ be defined by \eqref{e.at1}-\eqref{e.at3} with $t=0$. There exists $\epsilon=\epsilon_{L,N}>0$ such that \[\inf_{x\in C_{L}}\bP_{x}(A)>\epsilon.\]
\end{lem}
\begin{proof}
First, note that the event $A$ can be rewritten for the process $\rX$ (instead of $\rX-\overline{X}(0)$) by assuming that $\overline{X}(0)=0$. Hence, it is enough to show that there exists $\epsilon=\epsilon_{L,N}>0$ such that 
\[
\inf\{\bP_{x}(A):x\in C_{L}, \overline{x}=0\}>\epsilon. 
\] 
Fix any $x\in C_{L}$ with $\overline{x}=0$. We write $A$ as
\[A=A_{0,1}\cap E_1 \cap E_2 \cap E_3,\] where the events $E_i$ are defined as follows, using the same notation as in \eqref{e.at1}-\eqref{e.at3}. Define \begin{align*}
    &E_1\coloneqq \{X_j(1)\in B(-5e_1, \rad), \forall j\in G'\},\\&E_2\coloneqq \{X_j(1)\in B(5e_1, \rad), \forall j\in {D'}\},
\end{align*} and 
\[
E_3\coloneqq \{X_1(1)\in B(\gamma, \rad)\},
\]
where as before $\rad:=\frac{1}{4(N+1)}\in (0,1)$, and we recall that $\gamma$ is roughly the location of the barycenter of $G'\cup D'$. \jl{ Since we start from a deterministic configuration $x$, and conditioning on $A_{0,1}$ means there are no branching events between times $0$ and $1$, the positions of the particles at time $1$ are conditionally independent given $A_{0,1}$. Therefore,  
 \[
 \bP_{x}(A)=\bP_{x}(A_{0,1})\prod_{k=1}^3\bP_{x}\left(E_k\big|A_{0,1}\right).
\]
We can then bound each term in the last product as follows:} writing $\cN(x,I_d)$ for a $d$-dimensional Gaussian centered at $x$, 
\begin{align*}&\bP_{x}(E_1|A_{0,1})\geq \left(\min_{j\in G'}\bP\left(\cN(x_j,I_d)\in B(-5e_1,\rad)\right)\right)^{|G'|},\\&\bP_{x}(E_2|A_{0,1})\geq \left(\min_{j\in D'}\bP\left(\cN(x_j,I_d)\in B(5e_1,\rad)\right)\right)^{|D'|},\\&\bP_{x}(E_3|A_{0,1})\geq \bP\left(\cN(x_1,I_d)\in B(\gamma, \rad)\right),\end{align*} and $\bP_{x}(A_{0,1})=e^{-N}$. Taking the infimum over admissible configurations $x\in C_L$ satisfying $\overline{x}=0$, we obtain 
  \begin{align*}
  &\inf\{\bP_{x}(A):x\in C_L, \overline{x}=0\}\\
&\geq  e^{-N} \inf_{x\in C_L, \overline{x}=0}\left[\left(\min_{j\in G'}\bP\left(\cN(x_j,I_d)\in B(-5e_1, \rad)\right)\right)^{\left\lfloor\frac{N-1}{2}\right\rfloor}\right.\times\\&\times \left.\left(\min_{j\in D'}\bP\left(\cN(x_j,I_d)\in B(5e_1, \rad)\right)\right)^{\left\lceil\frac{N-1}{2}\right\rceil}\bP\left(\cN(x_1,I_d)\in B(\gamma, \rad)\right)\right]\\&\eqqcolon\epsilon_{L,N}>0. \qedhere
\end{align*} 
\end{proof} 
\begin{lem}\label{lem:lem4}
Let $B=B_0=B_{0,1}\cap B_{0,2}$ be defined by \eqref{e.bt1},\eqref{e.bt2}. There exists $\epsilon=\epsilon_{L,N}>0$ such that  \[\inf_{x\in C_{L}}\bP_{x}(B|A)>\epsilon.\]
\end{lem}
\begin{proof}
Let $\cS'\subset \cS$ be the set of possible values of $X$ at time $1$, given that $A$ occurred. Namely, $x\in \cS'$ if and only if
\begin{itemize}
\item[(a)] $C=\{1\}$, $G=\{2, \cdots, \left\lceil\frac{N+1}{2}\right\rceil\}$, and $D=\{\left\lceil\frac{N+1}{2}\right\rceil+1, \cdots, N\}$,
\item[(b)] $x_1\in B(\gamma, \rad)$,
\item[(c)] $x_j\in B(-5e_1, \rad)$ for all $j\in G$, and $x_j\in B(5e_1, \rad)$ for all $j\in D$.
\end{itemize}
In particular, we have that $G=G'$ and $D=D'$. By the Markov property and \eqref{e.at1}-\eqref{e.at3}, we have 
\begin{equation}\label{e.MarkovPropB}
\inf_{x\in C_L}\bP_x(B|A)\geq \inf_{x\in \cS'}\bP_x(\tilde B),
\end{equation}
where $\tilde B=\tilde{B}_{0,1}\cap \tilde{B}_{0,2}$ is the event $B$ ``taking place in the time interval $[0,1]$'' (rather than $[1,2]$).
Recalling the embedding of the BBB within a BBM as $X_j(t)=W_{\mathcal{I}_j(t)}(t)$, with $\mathcal{I}_j(0)=j$ for $j \in [N]$, 
we can decompose $\tilde{B}$ as the intersection of the following five events:
\begin{align*}
U_1\coloneqq&\bigcap_{j\in G'\cup D'}\{W_{j}(s) \textnormal{ doesn't branch from time $0$ until time $1$}\}\subset \tilde B_{0,1},\\
U_2\coloneqq &\{\textnormal{Each particle in $G'$ stays in $B(-5e_1, 2\rad)$ from time $0$ until time $1$}\}\subset \tilde B_{0,2} \\
U_3\coloneqq &\{\textnormal{Each particle in $D'$ stays in $B(5e_1, 2\rad)$ from time $0$ until time $1$}\}\subset \tilde B_{0,2},\\
U_4\coloneqq &\{\textnormal{The descendants of $W_1(0)$ stay in }B(\gamma, 2\rad)\textnormal{ for all times in }[0,1]\}\subset \tilde B_{0,2},
\\ U_5\coloneqq &\{\textnormal{$W_{1}(0)$ and its descendants branch at least } N-1\textnormal{ times}\textnormal{ in }[0,1]\}\subset \tilde B_{0,1}.\end{align*} 

Observe that for any $x\in \cS'$, we have
\jl{ 
\begin{equation}\label{e.probProduct}
\bP_x(\tilde B)=\left(\prod_{k=1}^3\bP_x(U_k|\cap_{j=1}^{k-1}U_j)\right)\bP_{x}(U_{4}\cap U_{5}|U_{1}\cap U_{2}\cap U_{3}).
\end{equation}}
We bound each factor in \eqref{e.probProduct} individually as follows. We have 
\[
\bP_x(U_1)\geq \bP_x(\text{a fixed particle doesn't branch in [0,1]})^{N-1}=e^{-(N-1)},
\]
Given $U_1$, the particles in $G'$ are independent until time $1$ or until they are killed, whichever comes first. Therefore, if $\rB=(B(s), s\geq 0)$ denotes a $d$-dimensional Brownian motion, then we have 
\[
\bP_x(U_2|U_1)\geq \min_{j\in G'}\bP(B(s)\in B(-5e_1, 2\rad), \forall s\in [0,1]|B(0)=x_j)^N.
\] 
Similarly, we have 
\[
\bP_x(U_3|U_1\cap U_2)\geq \min_{j\in D'}\bP(B(s)\in B(5e_1, 2\rad), \forall s\in [0,1]|B(0)=x_j)^N.
\]
\lou{Finally, for $x \in \cS'$, we have $x_1 \in B(\gamma,\rad)$, so $U_4 \cap U_5$ contains the event that, in the BBM, the particle $W_1(0)$ and its descendants branch at least $N-1$ times in the time interval $[0,1]$, and that all descendants of $W_1(0)$ stay within distance $\rad$ of $W_1(0)$ up to time $1$. By the branching property of the BBM, it follows that 
\begin{align*}
\bP_x(U_4 \cap U_5|U_1\cap U_2\cap U_3)
& \ge \bP_{W(0)=(0)}\Big(\cN(1) \ge N,\sup_{s \in [0,1],j \in \cN(s)}|W_j(s)| \le \rad\Big)\, ,
\end{align*}
where $\bP_{W(0)=(0)}$ denotes a probability measure under which $\rW=(W(s),s \ge 0)$ is a standard BBM started from a single particle at the origin. }
Combining these lower bounds, and taking the infimum over $x\in \cS'$ in \eqref{e.probProduct}, we obtain
\[
\inf_{x\in C_L}\bP_x(B|K)>\epsilon_{L,N},
\] for some $\epsilon_{L,N}>0$. The lemma then follows from \eqref{e.MarkovPropB}.
\end{proof}
\begin{proof}[Proof of Lemma~\ref{lem:lem6}]
Lemmas \ref{lem:lem3} and \ref{lem:lem4} imply that 
\[
    \inf_{x\in C_{L}}\bP_{x}(A\cap B)\geq \left(\inf_{x\in C_{L}}\bP_{x}(A)\right)\left(\inf_{x\in C_{L}}\bP_{x}(B|A)\right)> 0\, .\qedhere
\]
\end{proof}
\subsection{Proof of Lemma~\ref{lem:regen_unambiguous}}\label{subsec:deterministic}

This section is devoted to proving the deterministic Lemma~\ref{lem:regen_unambiguous}. For this we will use the following two claims.
\begin{claim}\label{claim:safe_branching}
Fix a configuration $x,w$ and an integer $\ell \in [N]$ with $w_\ell > 0$. Write $k=k(x,w,\ell)$ and $w^*=g(x,w,\ell)$. If $w^{*}_{k} \geq 1$ and $k \ne \ell$ then also $k(x,w^*,\ell) \ne \ell$. 
\end{claim}
\begin{proof}
Let
\[
y=\frac{1}{N-w_\ell-w_k}\sum_{i \in [N]\setminus \{k,\ell\}} w_ix_i\, ,
\]
and define 
\[
b = \frac{N-w_\ell-w_k}{N+1}y + \frac{(w_\ell+1)}{N+1}x_\ell + \frac{w_k}{N+1} x_k\, ;
\]
this is the barycenter of the configuration obtained from $x,w$ after letting $x_\ell$ branch (increasing its weight by $1$) but before any killing takes place (so the total weight is $N+1$ rather than $N$). 
Likewise, upon killing particle $x_{k}$ and letting $x_{\ell}$ branch consecutively, define 
\[
b^* = \frac{N-w_\ell-w_k}{N+1}y + \frac{(w_\ell^*+1)}{N+1}x_\ell + \frac{w_k^*}{N+1}x_k. 
\]
We have $(N-w_\ell-w_k)y=\sum_{i \in [N]\setminus \{k,\ell\}} w_ix_i = \sum_{i \in [N]\setminus \{k,\ell\}} w_i^*x_i$ since $w_i^*=w_i$ for all $i\in [N]\backslash\{k,l\}$ and $N-w^{*}_\ell-w^{*}_k=N-w_{\ell}-w_{k}$ since $w_\ell^*=w_\ell+1$ and $w_k^*=w_k-1$. Thus, the point $b^*$ is the barycenter of the configuration obtained from $x,w^*$ by letting $x_\ell$ branch but before any killing takes place, and also, 


\begin{equation}\label{e.rel_b_bstar}
b^*= \frac{N-w_{\ell}-w_{k}}{N+1}y + \frac{(w_\ell+2)}{N+1}x_\ell + \frac{(w_k-1)}{N+1}x_k
 = b + \frac{x_\ell-x_k}{N+1}. 
\end{equation}

As depicted in Figure~\ref{fig:claim1} below, since $|x_\ell-b| < |x_k-b|$, it follows immediately that $|x_\ell-b^*| < |x_k-b^*|$. 
\begin{figure}
  \centering
  \includegraphics[width=0.6\linewidth]{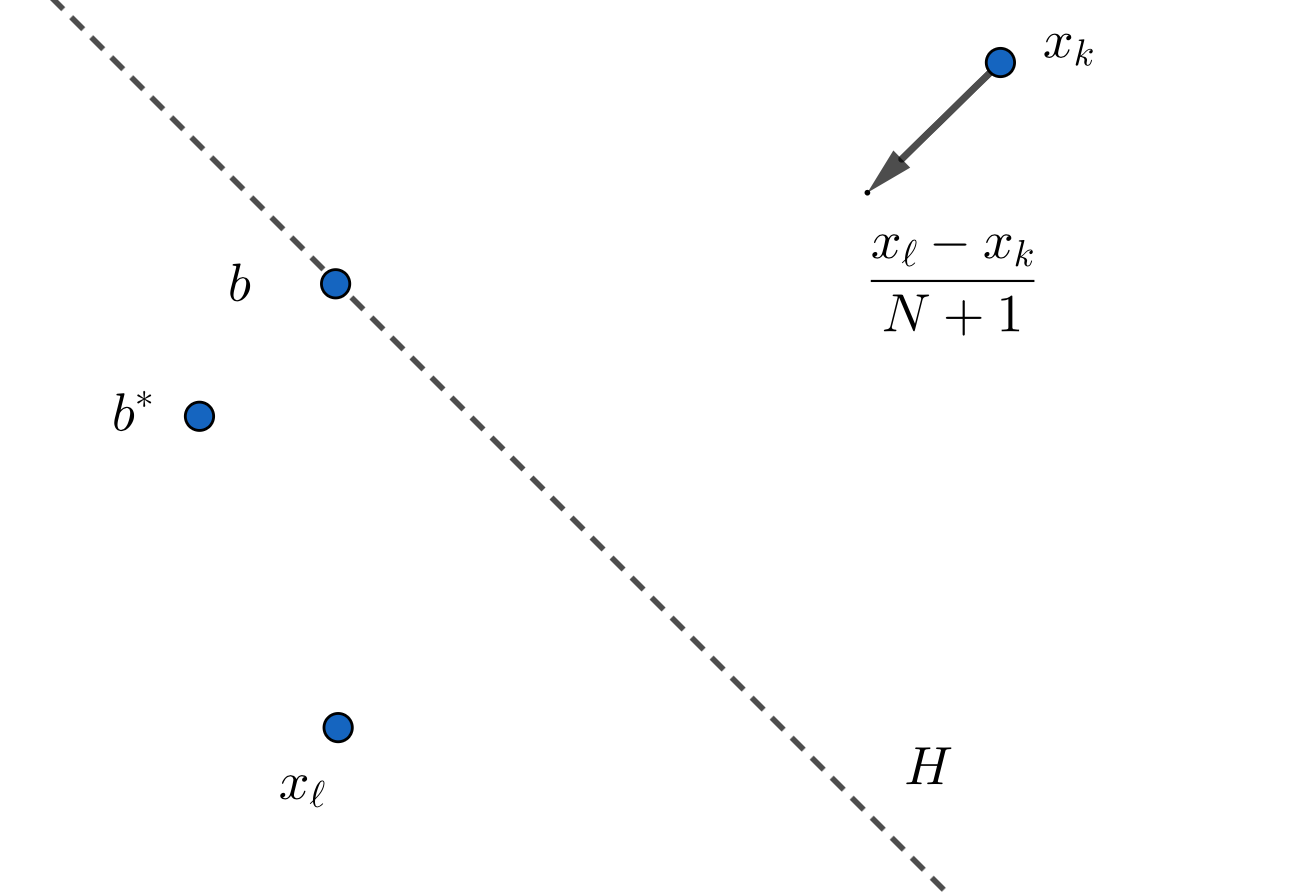}
  \caption{Pictorial description of (\ref{e.rel_b_bstar}) showing that $|x_\ell-b^*| < |x_k-b^*|$. Here, $H$ is the hyperplane perpendicular to $x_\ell-x_k$ and passing through $b$.}
  \label{fig:claim1} 
\end{figure} 
Since $w_k^*\geq 1$, this means there is at least one positive-weight particle further from $b^*$ than $x_\ell$, and thus $k(x,w^*,\ell)\ne\ell$. 
\end{proof}
\begin{claim}\label{claim:close_branching}
Fix a configuration $x,w$, let $b_0=N^{-1}\sum_{i=1}^N w_ix_i$, and let
\begin{equation}\label{eq:idef}
\ell\coloneqq \argmin_{i\in [N]:w_i>0}\left|x_i-b_0\right|. 
\end{equation}
If there exists $j \in [N]\setminus \{\ell\}$ with $w_j > 0$ then $k(x,w,\ell)\ne \ell$. 
\end{claim}
\begin{proof}
Since the killing rules for deterministic configurations are invariant under affine transformations, we may assume that $b_0=(0,0,\ldots,0)$ and $x_\ell=(1,0,\ldots,0)$. Writing 
\[
z=\frac{1}{N-w_\ell}\sum_{j \in [N]\setminus\{\ell\}} w_jx_j, 
\]
we have 
\[
b_0 = \frac{N-w_\ell}{N} z + \frac{w_\ell}{N} x_\ell, 
\]
so $z = (Nb_0-w_\ell x_\ell)/(N-w_\ell) = (-w_\ell/(N-w_\ell),0,\ldots,0)$.  
Since $z$ is the (weighted) barycenter of the particles excluding $x_\ell$, it follows that there is $j \ne \ell$ with $w_j > 0$ such that $x_j$ is an element of the halfspace $\mathcal{P}=\{(p_1,\ldots,p_d): p_1 \le -w_\ell/(N-w_\ell)\}$.

Now, similar to the proof of Claim~\ref{claim:safe_branching}, write 
\[
b= \frac{N-w_\ell}{N+1} z + \frac{w_\ell+1}{N+1} x_\ell
\]
for the barycenter of the configuration obtained from $x,w$ after letting $x_\ell$ branch but before any killing takes place. An easy calculation gives that $b=(1/(N+1),0,\ldots,0)$, so $|x_\ell-b| < |x_\ell-b_0|$. Moreover, for a particle $x_j \in \cP$ as in the previous paragraph, we have $|x_j-b| > |x_j-b_0|$. By our choice of $\ell$ we also have $|x_\ell-b_0| < |x_j-b_0|$, and thus $|x_\ell-b| < |x_j-b|$. Thus there is at least one particle further from $b$ than $x_\ell$, so $k(x,w,\ell)\ne \ell$. 
\end{proof}
We now conclude the proof as follows. 

\begin{proof}[Proof of Lemma~\ref{lem:regen_unambiguous}]
Fix any unambiguous configuration $x,w$ (recall \eqref{e.defunam}) with at least two particles with non-zero weights. 
For $b_{0}$ as in Claim \ref{claim:close_branching}, let $\ell = \argmin_{\ell\in [N]:w_\ell>0}\left|x_\ell-b_0\right|$. Let $w^{(0)}=w$, and for $j\geq 0$, $w^{(j+1)}=g(x, w^{(j)}, \ell)$, and $k_j=k(x,w^{(j)},\ell)$. By Claim \ref{claim:close_branching}, $k_0\not=\ell$. By Claim \ref{claim:safe_branching} and induction, for all $j\geq 0$ such that $w_{k_j}^{(j+1)}\geq 1$, we have $k(x, w^{(j+1)}, \ell)\not = \ell$, so $w_\ell^{(j+1)}=w_\ell^{(j)}+1$. Since the maximum possible weight is $N$, it follows that there is $j\leq N-1$ such that $w_{k_j}^{(j+1)}=0$. 


We have shown that, starting from any unambiguous configuration with at least two particles of non-zero weight, there exists a sequence of at most $N-1$ branching events which reduces the number of particles with nonzero weight by at least one. Repeating such a procedure at most $N-1$ times, we necessarily obtain a configuration with at most one nonzero weight. This completes the proof. 
\end{proof}

Given $\delta>0$, let $x\in \R^{d\times N}$ be a $\delta$-unambiguous configuration (as defined by \eqref{e.defdelta}) and introduce the notation
\begin{equation}\label{eq:point_neighbour_def}
\R^{d\times N}_{\delta/8, \:x}\coloneqq \left\{y\in \R^{d\times N}: y=(y_1, \cdots, y_N)\in \prod_{i=1}^NB\left(x_i, \frac{\delta}{8}\right)\right\},
\end{equation} \noindent which is the set of generic configurations within $\frac{\delta}{8}$ of $x$. The following corollary will be essential to the proof of Lemma~\ref{lem:lem10} in the next section (Section \ref{subsec:HarrisRec}).
\begin{cor}\label{c.sameSeq} 
Given any $\delta$-unambiguous configuration $x,w$, Lemma~\ref{lem:regen_unambiguous} assigns the same collapsing sequence of branching events to $y,w$ for all $\jl{y\in\R^{d\times N}_{\delta/8, \:x}}$.
\end{cor} 
\begin{proof} 
It is immediate from the definition of \jl{$\R^{d\times N}_{\delta/8, \:x}$} and the fact that $x$ is $\delta$-unambiguous that any element \jl{$y\in \R^{d\times N}_{\delta/8, \:x}$} is $\frac{\delta}{2}$-unambiguous. Moreover, $y_i\in \jl{B\left(x_i, \frac{\delta}{8}\right)}$ for each $i\in [N]$ so at any step in the proof of Lemma~\ref{lem:regen_unambiguous}, the barycenter of $y, w$ can only be at most \jl{$\frac{\delta}{8}$} away from the barycenter of $x, w$. Hence, Lemma~$\ref{lem:regen_unambiguous}$ assigns the same sequence to $y, w$ as to $x, w$. 
\end{proof} 
\subsection{Harris recurrence for the recentered BBB}\label{subsec:HarrisRec}
In this section we prove Lemma~\ref{lem:unambiguous_fast}, 
 Lemma~\ref{lem:lem10}, and finally Theorem~\ref{thm:HarrisRec}, which asserts the Harris recurrence of $\rX-\overline{\rX}$.  
 
 Recall that for $\delta>0$, we defined
$\xi_\delta \coloneqq \inf\{t\geq 0: X(t)\textnormal{ is }\delta\textnormal{-unambiguous}\}$, 
where the meaning of $\delta$-unambiguous is given by \eqref{e.defdelta}. 
Lemma~\ref{lem:unambiguous_fast} states that for $\delta$ small, the event $\{\xi_\delta<1\}$ occurs with probability as close to 1 as we wish, uniformly over starting configurations. 
\begin{proof}[Proof of Lemma~\ref{lem:unambiguous_fast}]
For each $\alpha\in (0,1)$, define the event
\[
D_\alpha\coloneqq \{\text{no branching occurs in }[0,\alpha]\}.
\]
For each $t\geq 0$ and $j\in [N]$, we define \[r_j(t)=|X_j(t)-\overline{X}(t)|.\] For any $x \in \R^{d \times N}$ and any $\delta>0$, we have 
\begin{align*}
\bP_x(\xi_\delta\geq 1|D_\alpha)&=\bP_x\left(\bigcap_{0\leq t<1}\left\{\min_{i\not = j\in [N]}|r_i(t)-r_j(t)|\leq \delta\right\}\Big|D_\alpha\right)\\&\leq \bP_x\left(\min_{i\not = j\in [N]}\left|r_i\left(\alpha\right)-r_j\left(\alpha\right)\right|\leq \delta \Big|D_\alpha\right) \\& \leq \sum_{i\not =j\in [N]}\bP_x\left(\left|r_i\left(\alpha\right)-r_j\left(\alpha\right)\right|\leq \delta\Big|D_\alpha\right)\\&\leq {N \choose 2} \max_{i\not= j\in [N]}\bP_x\left(\left|r_i\left(\alpha\right)-r_j\left(\alpha\right)\right|\leq \delta\Big|D_\alpha\right).
\end{align*}
Moreover, 
\begin{equation}\label{e.GaussianNormsBd}
\bP_x\left(\left|r_i\left(\alpha\right)-r_j\left(\alpha\right)\right|\leq \delta\Big|D_\alpha\right)=\bP_x\left(\left|g_{i,j}\left(X\left(\alpha\right)\right) -g_{j,i}\left(X\left(\alpha\right)\right) \right|\leq \delta \Big|D_\alpha\right),
\end{equation}
where 
\[
g_{i,j}\left(X\left(\alpha\right)\right)\coloneqq \left|\frac{N-1}{N}X_i\left(\alpha\right)-\frac{1}{N}X_j\left(\alpha\right)-\frac{1}{N}\sum_{k\in [N]\backslash\{i,j\}}X_k\left(\alpha\right)\right|. 
\]
is the distance between $X_i$ and $\overline{X}$ at time $\alpha$. Thus, by the tower law, we can write \eqref{e.GaussianNormsBd} as 
\[
\bE_x\left[\bP_x\left(\left|g_{i,j}\left(X\left(\alpha\right)\right)-g_{j,i}\left(X\left(\alpha\right)\right)\right|\leq \delta\Big|\sigma\left((X_k(s))_{k\in [N]\backslash\{i,j\}}:0\leq s\leq \alpha\right), D_\alpha\right)\Big|D_\alpha\right].
\]
Given $D_\alpha$, the particle positions at time $\alpha$ are independent Gaussians. So, as $\delta \downarrow 0$, the event in the conditional probability then requires that the norms of two independent Gaussians be arbitrarily close at time $\alpha$. It follows that
\begin{align}\label{e.term1Limit}
&\lim_{\delta\downarrow 0}\sup_{x\in \R^{d\times N}}\bP_x(\xi_\delta\geq 1|D_\alpha)\notag\\&=\lim_{\delta\downarrow 0}\sup_{x\in \R^{d\times N}}\max_{i\not= j\in [N]}\bP_x\left(\left|r_i\left(\alpha\right)-r_j\left(\alpha\right)\right|\leq \delta\Big|D_\alpha\right)\notag\\&=0.
\end{align}
Finally, we have 
\begin{align*}
\bP_x(\xi_\delta\geq 1)&=\bP_x(\xi_\delta\geq 1, D_\alpha)+\bP_x(\xi_\delta\geq 1, D_\alpha^c)\\&=e^{-\alpha N}\bP_x(\xi_\delta\geq 1|D_\alpha)+(1-e^{-\alpha N})\bP_x(\xi_\delta\geq 1|D_\alpha^c)\\&\leq e^{-\alpha N}\bP_x(\xi_\delta\geq 1|D_\alpha)+1-e^{-\alpha N}.
\end{align*}
Thus taking the supremum over $x\in \R^{d\times N}$, and subsequently the limits as $\delta\downarrow 0$ and as $\alpha\downarrow 0$, and using (\ref{e.term1Limit}), we deduce that 
\[
\lim_{\delta\downarrow 0}\sup_{x\in \R^{d\times N}}\bP_x(\xi_\delta\geq 1)=0.
\]
We conclude that for any $\epsilon>0$, there exists $\delta=\delta_\epsilon>0$ such that 
\[
\inf_{x\in \R^{d\times N}}\bP_x(\xi_\delta<1)>1-\epsilon.
\]
\end{proof}
We now turn to the proof of Lemma~\ref{lem:lem10}. Recall that the time-$t$ extent of the process is defined as 
\[ 
E(X(t))\coloneqq \max_{i\not =j\in [N]}|X_i(t)-X_j(t)| ; 
\] 
the lemma asserts that the extent reaches bounded values in a bounded time with high probability. For the proof we need one final definition. 
For any $t\geq s\geq 0$ and $\delta >0$, let 
\jl{ 
\begin{align*}
H_\delta(s,t]\coloneqq &
\left\{
\text{
For all $r \in (s,t]$ and $j \in [N]$, $X_j(r)$ has distance $<\delta$ from its time-$s$ ancestor
}
\right\}.
\end{align*}
We recall and emphasize that ancestral trajectories are continuous, and that the time-$s$ ancestor of $X_j(r)$ need not be $X_j(s)$.
}

{\color{cjl}
\begin{proof}[Proof of Lemma~\ref{lem:lem10}] 
First, for all $\lambda\geq1$, we have $T^{(2\lambda)}\leq T^{(2)}$. So, using the strong Markov property, we obtain for all $x\in \R^{d\times N}$
\[
\bP_x(T^{(2\lambda)}>2\lambda)\leq \bP_x(T^{(2\lambda)}>2\lfloor \lambda\rfloor) \leq \bP_x(T^{(2\lambda)}>2(\lfloor \lambda\rfloor-1))\left(\sup_{y\in \R^{d\times N}}\bP_y(T^{(2)}>2)\right).
\]
It follows by induction that 
\[
\bP_x(T^{(2\lambda)}>2\lambda)\leq \left(\sup_{y\in \R^{d\times N}}\bP_y(T^{(2)}>2)\right)^{\lfloor \lambda\rfloor},
\]
for all $x\in \R^{d\times N}$. Consequently, in order to conclude the proof of the lemma, it suffices to prove that
\begin{equation}\label{ineq:finiteExtent}
\sup_{x\in \R^{d\times N}}\bP_x(T^{(2)}>2)<1
\end{equation}

Next, fix $\epsilon\in (0,1)$. By Lemma~\ref{lem:unambiguous_fast}, there exists $\delta=\delta_\epsilon>0$ such that
\[
1-\epsilon<\inf_{x\in \R^{d\times N}}\bP_x(\xi_\delta<1). 
\] 
Since 
\[
\bP_x(T^{(2)}>2) \le \bP_x(T^{(2)}>2|\xi_\delta<1)\bP_x(\xi_\delta<1)+\bP_x(\xi_\delta \ge 1),
\]
to prove (\ref{ineq:finiteExtent}) it therefore suffices to show that 
\begin{equation}\label{ineq:finiteExtent2}
\sup_{x\in \R^{d\times N}} \bP_x(T^{(2)}>2|\xi_\delta < 1)<1\, .
\end{equation}

By Lemma~\ref{lem:regen_unambiguous}, there exists a deterministic sequence $(i_j)_{j=1}^m$ of $m \le (N-1)^2$ branching events that collapses the configuration $X(\xi_\delta), (1,1,\ldots, 1)$. For $0 \le s < t$ we define the the event 
\[
F(s,t]\coloneqq \{(i_j)_{j=1}^m\textnormal{ are the first $m$ branching events to occur in }(s,t]\}.
\] 
If $\xi_\delta < 1$, and $F(\xi_\delta,\xi_\delta+1]$ occurs and also $H_{\delta/8}(\xi_\delta,\xi_\delta+1]$ occurs, then $T^{(2)} \le \xi_\delta+1\le 2$ by Corollary~\ref{c.sameSeq}. Therefore, for all $x \in \R^{d \times N}$, 
\begin{align*}
\bP_x(T^{(2)}\le2|\xi_\delta < 1)
& \ge 
\bP_x(F(\xi_\delta,\xi_\delta+1] \cap H_{\delta/8}(\xi_\delta,\xi_\delta+1]|\xi_\delta < 1) \\
& \ge 
\inf_{y \in \R^{d \times N}}
\bP_y(F(0,1]\cap H_{\delta/8}(0,1])\, ,
\end{align*}
where at the last step we have invoked the strong Markov property at time $\xi_\delta$. The event in the final probability simply requires a fixed, bounded set of branching events to occur in order, and for particles to stay within a fixed distance of their time-zero ancestor for a fixed time, so clearly has positive probability (which in fact does not depend on the starting configuration $y$). This establishes (\ref{ineq:finiteExtent2}) and completes the proof. 
\end{proof}}

We will now prove Theorem \ref{thm:HarrisRec} using Lemma~\ref{lem:lem10}. We use $\Phi$ to denote the recentered BBB process, so that $\Phi_t\coloneqq X(t)-\overline{X}(t)$ for all $t\geq 0$. For each $t>0$ and $x\in \R^{d\times N}$, let $\mu_{t,x}\coloneqq \bP_x\circ\Phi_t^{-1}$ be the law of $\Phi_t$ given that $\Phi_0=x-\overline{x}$. Denote by $\varphi$ the $N$-fold product measure of a $d$-dimensional standard Gaussian and recall the notation 
\[
C_L\coloneqq \{x\in \R^{d\times N}:E(x)\leq L\}
\]
We begin with the following key estimate. 
\begin{lem}\label{lem:HarrisKeyPt}
For any $L\geq 0$, there exists $\gamma=\gamma(L,N)>0$ such that
\begin{equation}\label{e.lemharris}
\inf_{x\in C_L}\inf_{1\leq t\leq 2}\mu_{t,x}(A)\geq \gamma \varphi(A), \quad\text{for all Borel }A\subset \R^{d\times N}
\end{equation}
\end{lem}
\begin{proof}
Let $L\geq 0$. Define  
\[
D\coloneqq \{\text{no branching events occur in }[0,2]\}.
\] By monotonicity and independence, for any Borel $A \subset \R^{d\times N}$, we obtain 
\begin{align*}
\mu_{t,x}(A)&\geq \mu_{t,x}(A\cap D)\\&=\mu_{t,x}(A|D)\mu_{t,x}(D)\\&= e^{-2N}\int_{A} \prod_{i=1}^N \frac{e^{-\frac{|y_i-(x_i-\overline{x})|^2}{2t}}}{(2\pi t)^{d/2}}dy_1\cdots dy_N.
\end{align*} Moreover, for each $x\in C_L$, $y\in \R^{d\times N}$, and $t\in [1,2]$,
\jl{
\[
\frac{e^{-\frac{|y_i-(x_i-\overline{x})|^2}{2t}}}{(2\pi t)^{d/2}}
\geq \frac{e^{-\frac{|y_i|^2}{t}}e^{-\frac{|x_i-\overline{x}|^2}{t}}}{(2\pi t)^{d/2}}
\geq \frac{e^{-L^2}}{2^{d/2}}\frac{e^{-|y_i|^2}}{(2\pi)^{d/2}},
\] }
which implies that for each $x\in C_L$ and $t\in [1,2]$,
\jl{
\begin{align*}
e^{-2N}\int_{A} \prod_{i=1}^N \frac{e^{-\frac{|y_i-(x_i-\overline{x})|^2}{2t}}}{(2\pi t)^{d/2}}dy_1\cdots dy_N&\geq e^{-2N}\frac{e^{-NL^2}}{2^{dN/2}} \frac{1}{(2\pi)^{dN/2}}\int_A \prod_{i=1}^N e^{-|y_i|^2}dy_1\cdots dy_N \\&= e^{-2N} e^{-\frac{NL^2}{2}} \varphi(A).
\end{align*} }
Thus \eqref{e.lemharris} holds if we take \jl{$\gamma(L,N)\coloneqq e^{-2N}e^{-\frac{NL^2}{2}} >0$.}
\end{proof}
\begin{proof}[Proof of Theorem \ref{thm:HarrisRec}]
We wish to show that $\Phi$ is Harris recurrent. According to Definition \ref{dfn:HarrisRec}, it suffices to show that for the same choice of $\varphi$ as above, we have for all $A\subset \R^{d\times N}$ Borel, 
\[
\varphi(A)>0\Rightarrow \inf_{x\in \R^{d\times N}}\bP_x(\eta_A=\infty)=1,
\] where $\eta_A=\int_0^\infty \mathbbm{1}_{\{\Phi_t\in A\}}dt$. By Lemma~\ref{lem:lem10} and Remark \ref{r.HarrisRec} (i), the stopping times $(T_i; i\geq 1)$ defined by (\ref{eq:ti_def}) are finite $\bP_x$-almost surely. Fix $L\geq 0$ and any Borel set $A\subset \R^{d\times N}$. For each integer $i\geq 1$, define 
\[
\eta_A^i\coloneqq \int_{T_i+1}^{T_i+2}\mathbbm{1}_{\{\Phi_t\in A\}}dt.
\]
For all $i\geq 1$, the strong Markov property, Tonelli's Theorem and Lemma~\ref{lem:HarrisKeyPt} imply that 
\begin{align*}
\bE_x[\eta_A^i]&=\bE_x\left[\bE_{X(T_i)}\left[\eta_A^1\right]\right]\\&=\bE_x\left[\int_1^2\bP_{X(T_i)}(\Phi_t\in A)dt\right]\\&=\bE_x\left[\int_1^2\mu_{t,X(T_i)}(A)dt\right]\\&\geq \gamma \varphi(A).
\end{align*}
Let $\theta\coloneqq \gamma\varphi(A)\in (0,1)$, since $\gamma \in (0,1)$ and $\varphi$ is a probability measure. \jl{ Since $\eta_A^i\leq 1$ for each $i\geq 1$, it follows that
\begin{align*}
\theta&\leq \bE_x[\eta_A^i]\\&=\bE_x[\eta_A^i\mathbbm{1}_{\{\eta_A^i\leq \theta^2\}}]+\bE_x[\eta_A^i\mathbbm{1}_{\{\eta_A^i> \theta^2\}}]\\&\leq \theta^2\bP_x(\eta_A^i\leq \theta^2)+\bP_x(\eta_A^i>\theta^2)\\&=\theta^2+(1-\theta^2)\bP_x(\eta_A^i>\theta^2), 
\end{align*} 
for all $i\geq 1$}. Rearranging the above inequality, we now have 
\begin{equation}\label{e.PositiveProb}
\bP_x(\eta_A^i>\theta^2)\geq \frac{\theta-\theta^2}{1-\theta^2}=\frac{\theta}{1+\theta}>0, \;\;\;\forall i\geq 1.
\end{equation} 
Moreover, the same proof shows that almost surely 
\[
\bP_x(\eta_A^i>\theta^2\;|\; \cF_{T_{i}}) \ge 
\frac{\theta}{1+\theta}
\]
for all $i \ge 1$. Since the random variables $(\eta_A^j,1 \le j < i)$ are $\cF_{T_{i}}$-measurable, it follows that almost surely 
\[
\bP_x(\eta_A^i>\theta^2\;|\; \eta_A^j,1 \le j < i) \ge 
\frac{\theta}{1+\theta}.
\]
This implies that the collection of indicators $(\mathbbm{1}_{\{\eta_A^i>\theta^2\}};i \ge 1)$
stochastically dominates a sequence $(B_i;i\geq 1)$ of independent Bernoulli random variables with success probability $\frac{\theta}{1+\theta}$. 
It is immediate that, almost surely, $\eta_A^i>\theta^2$ for infinitely many $i$, and so 
\[
\eta_A \ge \sum_{i \ge 1} \eta_A^i \aseq \infty\, ,
\]
as required. \qedhere
\end{proof}
\noindent {\bf Remark.} It might seem that the random variables $(\eta_A^i;i \ge 1)$ ought in fact to be IID. If this were so, the preceding proof would simply consist of an application of the strong law of large numbers; but the situation is more subtle. As discussed in Remark~\ref{r:dependencies}, the law of $X(T_{i+1})-\overline{X}(T_{i+1})$ depends on the duration $T_{i+1}-T_i$, and similar issues create dependencies between the random variables $(\eta_A^i;i \ge 1)$.

\section{Open questions}\label{sec:conclusion}
\begin{itemize}
\item It would be nice to understand the behavior of the rescaling $\sigma(d,N)$ obtained in Theorem \ref{thm:main}, especially as $N\to\infty$. We have shown that $\sigma(d,1)=\sigma(d,2)=1$ for all $d\geq 1$, and that in general, $\sigma(d, N)$ is related to the quantity $\Sigma(d,N)=\bE_x[\tau_2-\tau_1]^{-1/2}Q$, where $Q$ is a $d$-by-$d$ matrix such that $C=QQ^T$ is the covariance matrix of $\Delta_1=X_1(\tau_2)-X_1(\tau_1)$. Thus, one way to gain information about $\sigma(d,N)$ would be to understand the relations between the coordinates of $\Delta_1$. This approach seems rather difficult to implement given the abstract nature of $\Delta_1$. 
However, it seems likely that 
\begin{equation}\label{e.limSigma}
\lim_{N\to\infty}\sigma(d,N)=0,
\end{equation} 
for all dimensions $d\geq 1$. 

\item For Borel $A \subset \R^d$, write 
\[
\pi_t^N(A)\coloneqq \frac{1}{N}\#\{\{X_1(t)-\overline{X}(t), \cdots, X_N(t)-\overline{X}(t)\}\cap A\}
\] 
for the empirical measure of the BBB process viewed from its barycentre. We expect that $\pi_t^N$ converges weakly as first $t \to \infty$, then $N \to \infty$, to a continuous Borel measure with compact support. It seems likely that the limiting measure is the one identified as the hydrodynamic limit of the branching particle systems studied in \cite{2020berestycki, 2020berestycki2}, where the authors consider $N$-BBMs with killing of the particle furthest from the origin. 
\end{itemize}
\section{Appendix}\label{sec:append}
Given two functions $f,g:\R\to \R$ with $g\not = 0$, we write $f\overset{m\to\infty}{\sim}g$ if $\frac{f(m)}{g(m)}=1+o(1)$ as $m\to\infty$. Moreover, we define the set 
\[
\Lambda\coloneqq \{\lambda :[0,1]\to[0,1] : \lambda \text{ is a strictly increasing, continuous bijection}\},
\]
and subsequently the Skorohod metric on $\cD([0,1], \R^d)$ by 
\[
d(x,y)\coloneqq \inf_{\lambda \in \Lambda}\{\max\left(d_\infty(\lambda , I), d_\infty(x, y\circ \lambda)\right)\},
\] where $I\in \Lambda$ is the identity map on $[0,1]$, and $d_\infty$ is the supremum metric defined by 
\[
d_\infty(f, g)\coloneqq \sup_{0\leq t\leq 1}|f(t)-g(t)|
\] for functions $f$ and $g$ on $[0,1]$. In particular, we observe that for all $x,y\in \cD([0,1], \R^d)$
\begin{equation}\label{e.metricOrder}
d(x,y)\leq d_\infty(x,y)
\end{equation} by taking $\lambda =I\in \Lambda$. We can now recall the multidimensional version of Donsker's invariance principle. 
\begin{thm}\emph{(Donsker's Theorem, Theorem 4.3.5 in \cite{MR1876437})}\label{Donsker}
Let $(\Delta_i;i\geq 1)$ be a sequence of independent and identically distributed $\R^d$-valued random variables with mean $\beta$ and $d$-by-$d$ covariance matrix $C=QQ^T$, for some matrix $Q$. Define the partial sums $S(j)\coloneqq \sum_{i=1}^{j}\Delta_i$. Then, we have 
\[
 \left(m^{-1/2}\left(S(\lfloor tm\rfloor)-  tm\beta\right),0\leq t\leq 1\right)\overset{d}{\to} (QB(t),0\leq t\leq 1)
 \]
 in the Skorohod topology on $\cD([0,1], \R^d)$ as $m\to\infty$, where $\rB=(B(t),0\leq t\leq 1)$ is a standard $d$-dimensional Brownian motion. 

\end{thm}
\subsection{Proof of Proposition \ref{prop:prop2}} 
Let $x\in \R^{d\times N}$ denote the initial configuration of $\rX$, i.e. $X(0)=x$. The main idea will be to approximate $\overline{\rX}$ by a sum of IID random variables, for which we may apply Donsker’s invariance principle. The approximation is justified by properties (1)-(5) in Proposition \ref{prop:prop2}. Throughout the appendix, we denote by $\rB=(B(t),t\geq 0)$ a standard $d$-dimensional Brownian motion and we fix $\mathfrak{m}\coloneqq \bE_x[\tau_2-\tau_1]$. \\ 
We first define the appropriate sum of IID random variables which automatically satisfies an invariance principle.
\begin{lem}\label{lem:lem13}
Consider the sequence of sums of increments $S(j)\coloneqq \sum_{i=1}^{j}\Delta_i$, where $\Delta_i=X_1(\tau_{i+1})-X_1(\tau_i)$. There exist $\beta\in \R^d$ and a $d$-by-$d$ matrix $Q$ such that \[\left(m^{-1/2}\left(S(\lfloor tm\rfloor)- tm\beta\right),0\leq t\leq 1\right)\overset{d}{\to} (QB(t),0\leq t\leq 1)\] in the Skorokod topology on $\cD([0,1], \R^d)$ as $m\to\infty$.
\end{lem}
\begin{proof}
By property (2), the random variables $(\Delta_i; i\geq 1)$ are IID. Moreover, by property (4), they belong to $L^2(\bP_x)$. Letting $C=QQ^T$ be the covariance matrix of $\Delta_1$ and taking $\beta=\bE_x[\Delta_1]$, the lemma then follows from Donsker's invariance principle (Theorem \ref{Donsker}).
\end{proof}

For each $s\geq \tau_1$, let $k[s]$ be the unique positive integer such that $\tau_{k[s]} \leq s< \tau_{k[s]+1}$ and for $s < \tau_1$ let $k[s]=0$. Note that each $k[s]$ is a random variable. We next control the errors in probability between an appropriate rescaling of $\rX_1$ and $S$.
\begin{lem}\label{lem:lem14}
Fix any initial configuration $x\in \R^{d\times N}$. We have
\[d_\infty\left(\frac{X_1(\cdot m)}{\sqrt{m}},\frac{S(k[\cdot m])}{\sqrt{m}}\right)\to 0\] in probability with respect to $\bP_x$ as $m\to\infty$.
\end{lem}
\begin{proof}
Since $\tau_1<\infty$ almost surely with respect to $\bP_x$ and $X$ is right-continuous, it suffices to prove that as $m\to\infty$
\[
m^{-1/2}\sup_{\tau_1\leq t\leq m}\left|X_1(t)-S(k[t])\right|\to 0
\]
in probability with respect to $\bP_x$. We have
\begin{align*}
    X_1(t)-S(k[t])&=X_1(t)-\sum_{i=1}^{k[t]}(X_1(\tau_{i+1})-X_1(\tau_i))\\&=X_1(t)-X_1(\tau_{k[t]+1})+X_1(\tau_1)
\end{align*} Therefore, 
\begin{align*} 
\sup_{\tau_1\leq t\leq m}|X_1(t)-S(k[t])|&=\sup_{\tau_1\leq t\leq m}|X_1(t)-X_1(\tau_{k[t]+1})+X_1(\tau_1)|\\&\leq \max_{1\leq i\leq k[m]}\sup_{\tau_i\leq s\leq \tau_{i+1}}|X_1(s)-X_1(\tau_i)|+|X_1(\tau_1)|
\end{align*} almost surely.
Fix any $\epsilon>0$. It follows that 
\begin{align} 
    &\bP_x\left(\sup_{\tau_1\leq t\leq m}m^{-1/2}|X_1(t)-S(k[t])|>\epsilon\right)\notag\\&\leq\bP_x\left(\max_{1\leq i\leq k[m]}\sup_{\tau_i\leq s\leq \tau_{i+1}}m^{-1/2}|X_1(s)-X_1(\tau_i)|>\frac{\epsilon}{2}\right)+\bP_x\left(m^{-1/2}|X_1(\tau_1)|>\frac{\epsilon}{2}\right)\label{e.splitProb}
    \end{align} 
Since $\tau_1<\infty$ almost surely with respect to $\bP_x$ and $X$ is right-continuous, then $|X_1(\tau_1)|<\infty$ almost surely with respect to $\bP_x$. Therefore, the second probability in \eqref{e.splitProb} converges to $0$ as $m\to\infty$. Consider the random variables \[Z_i\coloneqq \sup_{\tau_i\leq s\leq \tau_{i+1}}|X_1(s)-X_1(\tau_i)|, \;\;\;i=1, \cdots, k[m].\] By property (3) of Proposition \ref{prop:prop2}, the $(Z_i)_{i\geq 1}$ are identically distributed. Moreover, by property (4) of Proposition \ref{prop:prop2} and the triangle inequality, we obtain $\bE_x[Z_1^2]<\infty$. Since $(\tau_{i+1}-\tau_i;i\geq 1)$ are IID with finite mean by Property (1) of Proposition \ref{prop:prop2}, the strong law of large numbers implies that $\frac{\tau_n}{n}\to \mathfrak{m}$, $\bP_x$-a.s. as $n\to\infty$. It then easily follows that 
\begin{equation}\label{e.asymptoticK}
k[s]\overset{s\to\infty}{\sim} \mathfrak{m}^{-1}s, \quad \bP_x\text{-a.s.}
\end{equation} 
The result therefore follows from Fact \ref{l.convProb}.
\end{proof} 
Combining Lemma~\ref{lem:lem14} with property (5) of Proposition \ref{prop:prop2}, we obtain the following lemma.
\begin{lem}\label{lem.conv}
We have 
\[
d_\infty\left(\frac{\overline{X}(\cdot m)}{\sqrt{m}}, \frac{S(k[\cdot m])}{\sqrt{m}}\right) \to 0
\] in probability as $m\to\infty$. In particular, by (\ref{e.metricOrder}), this implies that \[d\left(m^{-1/2}\left(\overline{X}(\cdot m)-k[\cdot m]\beta\right), m^{-1/2}\left(S(k[\cdot m])-k[\cdot m]\beta\right)\right)\to 0\] in probability as $m\to\infty$, where $\beta\coloneqq \bE_x[\Delta_1]$.
\end{lem}
\begin{proof}
For every $\epsilon>0$, the triangle inequality implies that
\begin{align*}
&\bP_x\left(d_\infty\left(\frac{\overline{X}(\cdot m)}{\sqrt{m}}, \frac{S(k[\cdot m])}{\sqrt{m}}\right)>\epsilon\right)\\&\leq \bP_x\left(d_\infty\left(\frac{\overline{X}(\cdot m)}{\sqrt{m}}, \frac{X_1(\cdot m)}{\sqrt{m}}\right)>\frac{\epsilon}{2}\right)+\bP_x\left(d_\infty\left(\frac{X_1(\cdot m)}{\sqrt{m}}, \frac{S(k[\cdot m])}{\sqrt{m}}\right)>\frac{\epsilon}{2}\right)\to 0
\end{align*}
as $m\to\infty$.
\end{proof}

We will make use of the following version of Slutsky's Theorem for general metric space-valued random variables. 
\begin{thm}\emph{(Slutsky's Theorem, Theorem 3.1, \cite{B})} \label{Slutsky}
Let $(M,d)$ be a metric space. Suppose that $(X_n,Y_n)$ are random elements of $M\times M$. If $X_n\xrightarrow{d} X$ and $d(X_n,Y_n)\to 0$ in probability, then $Y_n\xrightarrow{d} X$. 
\end{thm} We are now able to prove Proposition \ref{prop:prop2}. 
\begin{proof}[Proof of Proposition \ref{prop:prop2}] 
In this proof, every convergence in distribution statement is with respect to the Skorohod topology on $\cD([0,1], \R^d)$. Lemma \ref{lem:lem13} and \eqref{e.asymptoticK} imply that 
\begin{align*}
&\left(m^{-1/2}\left(S(k[tm])-k[tm]\beta\right),0\leq t\leq 1\right)\\&\overset{m\to\infty}{\sim}\left(m^{-1/2}\left(S\left(\big\lfloor tm\mathfrak{m}^{-1}\big\rfloor\right)- tm\mathfrak{m}^{-1} \beta\right),0\leq t\leq 1\right)\\&\xrightarrow{d}\left(\mathfrak{m}^{-1/2}QB(t):0\leq t\leq 1\right)
\end{align*} 
as $m\to\infty$. Combining this with Lemma~\ref{lem.conv}, we can apply Slutsky's Theorem (Theorem \ref{Slutsky}) to obtain 
\[
\left(m^{-1/2}\left(\overline{X}(tm)-k[tm]\beta\right), 0\leq t\leq 1\right)\xrightarrow{d} \left(\mathfrak{m}^{-1/2}QB(t), 0\leq t\leq 1\right)
\]  
as $m\to\infty$. Hence 
\[
\left(m^{-1/2}\left(\overline{X}(tm)-tm\alpha\right), 0\leq t\leq 1\right)\xrightarrow{d} (\Sigma B(t),0\leq t\leq 1)
\] 
as $m\to\infty$, where in view of \eqref{e.asymptoticK}, we have defined 
\begin{equation*}
\alpha \coloneqq \beta\mathfrak{m}^{-1}=\bE_x[\Delta_1]\mathfrak{m}^{-1}\quad\text{and}\quad \Sigma\coloneqq \mathfrak{m}^{-1/2}Q.
\qedhere\end{equation*}
\end{proof}

\addtocontents{toc}{\SkipTocEntry} 
\section{Acknowledgements}
The authors would all like to thank Julien Berestycki and Sarah Penington for helpful discussions and correspondence, in particular regarding their related work; and to thank Jeremy Quastel for initially proposing the study of this model in 2012. The authors would also like to thank an anonymous referee for an extremely careful reading that substantially improved the paper.

LAB was partially supported by NSERC Discovery Grant 643473 and Discovery Accelerator Supplement 643474, and FRQNT Grant 206470. JL was partially supported by NSERC Discovery Grant 247764, FRQNT Grant 250479, and the Canada Research Chairs program. TT was partially supported by FRQNT Grant 250479. 


\small 



%
%
%
%
%
\normalsize
                 
\appendix

\end{document}